\documentclass[twoside,11pt]{article}

\usepackage{jmlr2e}

\ShortHeadings{Rates of Greedy and Random Quasi-Newton Methods}{Dachao Lin, Haishan Ye and Zhihua Zhang}
\firstpageno{1}

\usepackage{microtype}
\usepackage{graphicx}

\usepackage{amsmath,amsfonts,bm}

\def\1{\bm{1}}

\def\vb{{\bm{b}}}
\def\vc{{\bm{c}}}

\def\ve{{\bm{e}}}

\def\vh{{\bm{h}}}

\def\vu{{\bm{u}}}
\def\vv{{\bm{v}}}
\def\vw{{\bm{w}}}
\def\vx{{\bm{x}}}
\def\vy{{\bm{y}}}
\def\vz{{\bm{z}}}

\def\mA{{\bm{A}}}
\def\mB{{\bm{B}}}
\def\mC{{\bm{C}}}

\def\mG{{\bm{G}}}
\def\mH{{\bm{H}}}
\def\mI{{\bm{I}}}
\def\mJ{{\bm{J}}}

\def\mL{{\bm{L}}}

\def\mQ{{\bm{Q}}}
\def\mR{{\bm{R}}}

\def\mU{{\bm{U}}}

\def\mW{{\bm{W}}}
\def\mX{{\bm{X}}}

\DeclareMathAlphabet{\mathsfit}{\encodingdefault}{\sfdefault}{m}{sl}
\SetMathAlphabet{\mathsfit}{bold}{\encodingdefault}{\sfdefault}{bx}{n}

\def\gD{{\mathcal{D}}}

\def\gM{{\mathcal{M}}}
\def\gN{{\mathcal{N}}}

\def\gS{{\mathcal{S}}}

\def\gV{{\mathcal{V}}}

\def\sP{{\mathbb{P}}}

\def\sR{{\mathbb{R}}}

\newcommand{\E}{\mathbb{E}}

\def\tr{\mathrm{tr}}
\usepackage[utf8]{inputenc} 
\usepackage[T1]{fontenc}    
\usepackage{url}            
\usepackage{booktabs}       
\usepackage{nicefrac}       
\usepackage{color}          
\usepackage{mathrsfs}
\usepackage{algorithm}
\usepackage{algorithmic}
\usepackage{bbm}
\usepackage{bm}
\usepackage{amsfonts, amsmath, amssymb} 
\usepackage{float}
\usepackage{diagbox}
\usepackage{subcaption}
\usepackage{breakcites}
\usepackage[dvipsnames]{xcolor}
\usepackage{grffile}
\usepackage{makecell}
\usepackage{multirow}

\def\diag{\mathrm{diag}}
\newcommand{\norm}[1]{\left\|#1\right\|}

\def\tr{\mathrm{tr}}
\def\Broyd{\mathrm{Broyd}}
\def\BFGS{\mathrm{BFGS}}
\def\DFP{\mathrm{DFP}}

\usepackage{lastpage}
\jmlrheading{23}{2022}{1-\pageref{LastPage}}{10/21; Revised
5/22}{5/22}{21-1282}{Dachao Lin, Haishan Ye and Zhihua Zhang}
\ShortHeadings{Explicit Convergence Rates of Greedy and Random Quasi-Newton Methods}{Lin, Ye and Zhang}

\begin{document}

\title{Explicit Convergence Rates of Greedy and Random Quasi-Newton Methods}  

\author{\name Dachao Lin
	\email lindachao@pku.edu.cn \\
	\addr Academy for Advanced Interdisciplinary Studies \\ 
	Peking University \\
	Beijing, China
	\AND
	\name Haishan Ye\thanks{Corresponding Author.} \email yehaishan@xjtu.edu.cn \\
	\addr School of Management \\ 
	Xi'an Jiaotong University \\
    Xi'an, China
	\AND
	\name Zhihua Zhang \email zhzhang@math.pku.edu.cn \\
	\addr School of Mathematical Sciences \\
	Peking University \\
	Beijing, China
}

\editor{Philipp Hennig}
\maketitle

\begin{abstract}%
    Optimization is important in machine learning problems, and quasi-Newton methods have a reputation as the most efficient numerical methods for smooth unconstrained optimization. 
	In this paper, we study the explicit superlinear convergence rates of quasi-Newton methods and address two open problems mentioned by \citet{rodomanov2021greedy}.
	First, we extend \citet{rodomanov2021greedy}'s results to random quasi-Newton methods, which include common DFP, BFGS, SR1 methods. 
	Such random methods employ a random direction for updating the approximate Hessian matrix in each iteration.
	Second, we focus on the specific quasi-Newton methods: SR1 and BFGS methods.
	We provide improved versions of greedy and random methods with provable better explicit (local) superlinear convergence rates.
	Our analysis is closely related to the approximation of a given Hessian matrix, unconstrained quadratic objective, as well as the general strongly convex, smooth, and strongly self-concordant functions. 
\end{abstract}

\begin{keywords}
	quasi-Newton methods, superlinear convergence, local convergence, rate of convergence, Broyden family, SR1, BFGS, DFP
\end{keywords}

\section{Introduction}
Many machine learning problems can be formulated to the minimization of an objective defined as the expectation over a set of random functions
\citep{liu1989limited, bottou2005line, shalev2008svm, mokhtari2014quasi, mokhtari2015global}.
Specifically, given the training sample $\vz \sim \gD$, where $\gD$ is the data distribution, we consider an optimization function $f: \sR^d \to \sR$:
\[ \min_{\vx\in\sR^d} f(\vx) = \E_{\vz\sim \gD} \; \ell(\vx, \vz) + R(\vx), \]
where $\ell(\vx, \vz)$ is the loss with respect to the training sample $\vz$, and $R(\vx)$ is some regularization function, such as $\norm{\vx}^2_2$.
When $\gD$ is the empirical distribution of training samples $\{\vz_i\}_{i=1}^n$, we could recover the classical finite-sum empirical risk minimization:
\[ \min_{\vx\in\sR^d} f(\vx) = \frac{1}{n}\sum_{i=1}^n \ell(\vx, \vz_i) + R(\vx). \]
Such a finite-sum formulation encapsulates a wide variety of machine learning problems including least squares regression, support vector machines (SVM), logistic regression, neural networks, and graphical models.

Previous methods mainly use first-order methods by evaluating objective function gradients $\nabla f(\vx)$, such as gradient descent, stochastic gradient descent, accelerated gradient descent \citep{nesterov2003introductory}, Adagrad \citep{duchi2011adaptive}, Adam \citep{kingma2014adam}, etc.
These methods dominate the current optimization methods of machine learning problems, and have affordable computation complexity in each iteration.  
However, these first-order methods generally only have a linear or sublinear convergence rate even if the objective has nice properties.

Recently, second-order methods have also received great attention due to their fast convergence rates compared to first-order methods. 
But second-order methods, such as Newton's method, are impractical because the exact Hessian matrix $\nabla^2 f(\vx)$ needs high computation cost in general cases.
Common wisdom proposes quasi-Newton methods by replacing Hessian matrices with some reasonable approximations. 
The approximation is updated in iterations based on some special formulas from the previous variation.

Quasi-Newton methods have a broad application in machine learning problems \citep{Bordes2009, yu10a, mokhtari2015global, yuan2020convergence, ye2020, Liu2021}.
There exist various quasi-Newton algorithms with different Hessian approximations. 
The three most popular versions are the \textit{Davidon-Fletcher-Powell (DFP) method} \citep{fletcher1963rapidly, davidon1991variable}, the \textit{Broyden-Fletcher-Goldfarb-Shanno (BFGS) method} \citep{broyden1970convergence2, broyden1970convergence,fletcher1970new, goldfarb1970family, shanno1970conditioning}, and the \textit{Symmetric Rank 1 (SR1) method} \citep{broyden1967quasi,davidon1991variable}, all of which belong to the Broyden family \citep{broyden1967quasi} of quasi-Newton algorithms.
The most attractive property of quasi-Newton methods compared to the classical first-order methods, is their superlinear convergence, which can trace back to the 1970s \citep{powell1971convergence, broyden1973local, dennis1974characterization}. 
However, the superlinear convergence rates provided in prior work are asymptotic \citep{stachurski1981superlinear, griewank1982local, byrd1987global, yabe1996local, kovalev2020fast}. 
The results only show that the ratio of successive residuals tends to zero as the running iterations approach to infinity, i.e.,
\[ \lim_{k \to +\infty} \frac{\norm{\vx_{k+1}-\vx_*}}{\norm{\vx_k-\vx_*}}=0, \text{ or } \norm{\vx_{k+1}-\vx_*}=o(\norm{\vx_{k}-\vx_*}), \]
where $\{\vx_k\}$ is iterative update sequence, $k$ is the iteration counter, and $\vx_*$ is the optimal solution.
It is unknown whether the residuals converge like $O(c^{k^2}), O(k^{-k})$, where $c\in(0, 1)$ is some constant.
Hence, the theory is inadequate and there still lacks of a specific superlinear convergence rate.
Additionally, machine learning problems have requirement of the explicit convergence rates to compare the performance and design better algorithms for applications. 
Therefore, \emph{to give a better guidance of quasi-Newton methods in machine learning problems, we are still interested in the explicit rates of quasi-Newton methods.}

Recently, \citet{rodomanov2021greedy} gave the first explicit local superlinear convergence for their proposed new quasi-Newton methods. 
They introduced greedy quasi-Newton updates by greedily selecting from basis vectors to maximize a certain measure of progress, and established an explicit non-asymptotic bound on the local superlinear convergence rate correspondingly.
However, as \citet{rodomanov2021greedy} stated, ``greedy methods require additional information beyond just the gradient of the objective function.''
A natural idea might be to replace the greedy strategy with a randomized one. 
Indeed, the strategy of randomness  has almost the same performance as the greedy one, which has been observed in \citet{rodomanov2021greedy}'s experiments. 
Therefore, one can expect that it should be possible to establish similar theoretical results about its superlinear convergence, but they did not provide theoretical guarantees. 
This raises the issue:
\emph{can we give explicit superlinear rates for random quasi-Newton methods theoretically?}
In addition, \citet{rodomanov2021greedy}'s proofs are mainly applicable to the DFP methods because they reduced all possible Broyden family to the DFP update based on the monotonicity property (see Lemma \ref{lemma:monotonic}).
However, the SR1 and BFGS updates are more popular and faster than the DFP update in practice, which also has been verified in their experiments.
Thus, it is natural to ask 
\emph{can we obtain separate superlinear rates for different quasi-Newton methods?}

In this work, we solve the above two problems rigorously. We extend \citet{rodomanov2021greedy}'s results into random quasi-Newton methods, and improve the local superlinear convergence rates by our revised greedy or random SR1 and BFGS methods. 
We present our contribution in detail as follows:
\begin{itemize}
	\item First, we extend \citet{rodomanov2021greedy}'s results to random quasi-Newton methods, which use a random direction for updating the approximate Hessian matrix. Our superlinear convergence rate is of the form  $(1-\frac{1}{d\varkappa+1})^{k(k-1)/2}$ with high probability, which is similar as the greedy-type methods proposed by \citet[Theorem 4.9]{rodomanov2021greedy}. Here, $\varkappa$ is the condition number of the objective function, $k$ is the current iteration, and $d$ is the dimension of parameters. 
	\item Second, for specific quasi-Newton methods, including SR1 and BFGS methods, we provide improved versions of greedy and random methods.
	We show that for approximating a fixed Hessian matrix, both the methods share a faster condition-number-free convergence. 
	Particularly, we can obtain the superlinear convergence rate $O((1{-}\frac{k}{d})_+)$ for the SR1 update, and the linear convergence rate $O((1{-}\frac{1}{d})^k)$ for the BFGS update, where $(x)_+=\max\{x,0\}$. 
	Both the findings improve the original convergence rate $O((1{-}\frac{1}{d \varkappa})^k)$ by \citet[Theorem 2.5]{rodomanov2021greedy}.
	\item Third, we extend our analysis to a practical scheme, showing (local) superlinear convergence under our proposed greedy/random SR1 and BFGS update, when applied to unconstrained quadratic objective or strongly self-concordant functions. 
	We list our results in Table \ref{table:res} with the same formulation as the work of \cite{rodomanov2021greedy}. 
	Note that in general, the convergence goes through two phases. 
	The first phase lasts for $k_0$ iterations, and only has a linear convergence rate $O((1{-}\frac{1}{2\varkappa})^{k_0})$. 
	The second phase has a superlinear convergence rate $O((1{-}\frac{1}{d})^{k(k-1)/2})$.
	Our revised bound takes fewer first-phase iterations $k_0$ as well as a faster (condition-number-free) superlinear convergence rate in the second phase compared to \cite{rodomanov2021greedy}'s results. 
\end{itemize}

\begin{table*}[t] 	
	\centering
	\begin{tabular}{ccc}
		\toprule
		Quasi-Newton Methods & Superlinear Rates & $k_0$\\
		\midrule
		\makecell{Greedy Broyden \\ \cite[]{rodomanov2021greedy}} & $\left(1-\frac{1}{d\varkappa}\right)^{k(k-1)/2}\left(\frac{1}{2}\right)^k\! \left(1{-}\frac{1}{2\varkappa}\right)^{k_0}$ & $O\left(d\varkappa\ln(d\varkappa)\right)$ \\  
		\midrule
		\makecell{Random Broyden \\ (Corollary \ref{cor:random-gel})} & $\big(1{-}\frac{1}{d\varkappa{+}1}\big)^{k(k{-}1)/2}\left(\frac{1}{2}\right)^k\! \left(1{-}\frac{1}{2\varkappa}\right)^{k_0}$ & $O\left(d\varkappa\ln(d\varkappa/\delta)\right)$ \\  
		\midrule
		\makecell{Greedy BFGS*/SR1 \\ (Corollary \ref{cor:gel-sr1-bfgs})} & $\left(1-\frac{1}{d}\right)^{k(k-1)/2}\left(\frac{1}{2}\right)^k\! \left(1{-}\frac{1}{2\varkappa}\right)^{k_0}$  & $O\left(\left(d+\varkappa\right) \ln(d\varkappa)\right)$ \\
		\midrule
		\makecell{Random BFGS/SR1 \\ (Corollary \ref{cor:gel-sr1-bfgs})} & $\big(1{-}\frac{1}{d+1}\big)^{k(k-1)/2}\left(\frac{1}{2}\right)^k\! \left(1{-}\frac{1}{2\varkappa}\right)^{k_0}$  & $O\left(\left(d+\varkappa\right) \ln(d\varkappa/\delta)\right)$ \\
		\bottomrule
	\end{tabular}
\caption{Comparison of the existing specific superlinear convergence rates of the random or greedy quasi-Newton methods in the view of $\lambda_f(\cdot)$ (shown in Eq.~\eqref{eq:lambda}) under strongly self-concordant objective, where $d$ is the dimension of parameters, $\varkappa$ is the condition number of the objective function, $k_0$ is the iteration number last for the first phase and $k$ is the iteration number of the subsequent second phase. 
For the randomized methods, the presented rates hold with probability at least $1-\delta$. 
(*): Our greedy BFGS method is not practical.}
	\label{table:res}
\end{table*}

\subsection{Other Related Work}
In addition to the work of \citet{rodomanov2021greedy}, there are other results of explicit local superlinear convergence analysis along this line of research.
\citet{rodomanov2021rates}  analyzed the well-known DFP and BFGS methods, which are based on a standard Hessian update direction through the previous variation. They demonstrated faster initial convergence rates, while slower final rates compared to \citet{rodomanov2021greedy}'s results.
\citet{Rodomanov_2021} improved \citet{rodomanov2021rates}'s results by reducing the dependence of the condition number $\varkappa$ to $\ln\varkappa$, though having similar worse long-history behavior.
\citet{jin2020non} provided a non-asymptotic dimension-free superlinear convergence rate of the original Broyden family when the initial Hessian approximation is also good enough. However,  the two issues mentioned earlier remain open. 

The remainder of this paper is organized as follows. 
We present preliminaries in Section \ref{sec:prel}, and discuss the rates of random quasi-Newton methods in Section \ref{sec:random}.
In Section \ref{sec:faster}, we show faster superlinear convergence rates of our revised greedy/random SR1 and BFGS methods.
Then in Section \ref{sec:compa}, we show comparison with the work of \cite{rodomanov2021greedy} in detail. We give some empirical results in Section~\ref{sec:exper}.
Finally,  we conclude our results in Section \ref{sec:conclude}. 				
\section{Preliminaries}\label{sec:prel}

First of all, we present some  notation.  
We denote vectors by lowercase bold letters (e.g., $ \vu, \vx$), and matrices by capital bold letters (e.g., $ \mW = [w_{i j}] $).
We use $\ve_1, \ldots, \ve_d$ for the $d$-dimensional standard coordinate directions,
and $(x)_+=\max\{x, 0\} $ for $x \in \sR$. 
Let $\lambda_{\max}(\mA) = \lambda_1(\mA)\geq \dots \geq \lambda_d(\mA)$ be the eigenvalues of a real symmetric matrix $\mA\in\sR^{d\times d}$, and $\|\cdot\|$ denotes the standard Euclidean norm ($\ell_2$-norm) for vectors, or induced $\ell_2$-norm (spectral norm) for a given matrix: $\norm{\mA} = \sup_{\norm{\vu}=1, \vu\in\sR^d}\norm{\mA\vu}$. 
We denote $\gS^{d-1} := \{\vx \in \sR^d: \|\vx\|=1\}$ as the standard Euclidean sphere in $\sR^d$, and $\mathrm{Unif}(\gS^{d-1})$ as the uniform distribution from $\gS^{d-1}$. We use $\gN(\bm{0}, \mI_d)$ as the standard Gaussian distribution, where $\mI_d\in \sR^{d \times d}$ is the identity matrix.

For two symmetric matrices $\mA$ and $\mB \in \sR^{d \times d}$, we denote $\mA \succeq \mB$ (or $\mB \preceq \mA$) if $\mA-\mB$ is a positive semi-definite matrix, and $\mA \succ \mB$ (or $\mB \prec \mA$) if $\mA-\mB$ is a positive definite matrix.
Following \citet{rodomanov2021greedy}'s notation,
for a given positive definite matrix $\mA$ (i.e., $\mA \succ 0$), we induce a pair of conjugate Euclidean norms: $\| \vx \|_{\mA}  := \sqrt{\vx^\top \mA \vx}$ and  $\| \vx \|_{\mA}^* := \sqrt{\vx^\top \mA^{-1} \vx}$. 
When $\mA = \nabla^2 f(\vx) \succ 0$ for some $\vx \in \sR^d$, we prefer to use notation $\|\cdot\|_{\vx}$ and $\|\cdot\|_{\vx}^*$, provided that there is no ambiguity with the reference function $f$.

Next, we introduce some common definitions used in this paper below.
\begin{definition}[Strongly convex and smooth]
	A twice differentiable function $f\colon \sR^{d}\to \sR$ is $\mu$-strongly convex and $L$-smooth ($\mu, L > 0$), if
	\[ \mu \mI_d \preceq \nabla^2 f(\vx) \preceq L \mI_d, \; \forall \vx\in\sR^d. \]
    Additionally, the condition number of a $\mu$-strongly convex and $L$-smooth function is $\varkappa := L/\mu$.
\end{definition}

We also need the same assumption of \textit{strongly self-concordancy} followed by \citet{rodomanov2021greedy}. And \citet[Section 4]{rodomanov2021greedy} have already mentioned several properties and examples of strongly self-concordant functions, such as a strongly convex function with Lipschitz continuous Hessians.
\begin{definition}[Strongly self-concordant]
	A twice differentiable function $f\colon \sR^{d}\to \sR$ is $M$-strongly self-concordant ($M>0$), if the Hessians are close to each other in the sense that 
	\[ \nabla^2 f(\vy)-\nabla^2 f(\vx) \preceq M\|\vy-\vx\|_{\vz} \nabla^2 f(\vw), \ \forall \vx, \vy, \vz, \vw \in \sR^d. \]
\end{definition}

Finally, we recall the rate of convergence used in this paper.
\begin{definition}[R-Linear/Superlinear convergence] 
Suppose a scalar sequence $\{x_k\}$ converges to $0$ with
	\[ \lim_{k \to +\infty} \frac{|x_{k+1}|}{|x_k|} = q \in [0, 1). \]
	Now suppose another sequence $\{y_k\}$ converges to $y^*$ and satisfies that $|y_k-y^*|\leq |x_k|, \forall k \geq 0$.
	We say $\{y_k\}$ converges superlinearly if and only if $q=0$, linearly if and only if $q \in (0, 1)$.
\end{definition}

\subsection{Notation for Convergence Analysis}

For convergence analysis, we introduce two measures which describe the approximation precision of the positive definite matrices:
\begin{equation}\label{eq:sigmaA}
    \sigma_{\mA}(\mG) := \tr\left[\left(\mG-\mA\right)\mA^{-1}\right] = \tr\left(\mG\mA^{-1}\right)-d, \text{ where } \mG\succeq \mA \succ \bm{0},
\end{equation}
and
\begin{equation}\label{eq:tauA}
\tau_{\mA}(\mG):=\tr(\mG-\mA), \text{ where } \mG\succeq \mA \succ \bm{0}.
\end{equation}
Moreover, we estimate the convergence rate of a strongly convex objective $f(\vx)$ by the local norm of the gradient:
\begin{equation}\label{eq:lambda}
\lambda_f(\vx) := \norm{\nabla f(\vx)}_{\vx}^* = \sqrt{\nabla f(\vx)^\top [\nabla^2 f(\vx)]^{-1} \nabla f(\vx)}, \; \vx\in\sR^d.
\end{equation}
Note that $\sigma_{\mA}(\mG)$ and $\lambda_f(\vx)$ are also introduced in the work of  \cite{rodomanov2021greedy}.
When applied to the update sequences $\{\vx_k\}$ and $\{\mG_k\}$ from a specific algorithm, we also denote the following notation for brevity:
\begin{equation}\label{eq:lam-sig}
	\lambda_k:=\lambda_f(\vx_k), \; \sigma_k := \sigma_{\nabla^2 f(\vx_k)}(\mG_k) \; \mbox{ and } \;  \tau_k := \tau_{\nabla^2 f(\vx_k)}(\mG_k).
\end{equation}

\subsection{Quasi-Newton Updates}
Before starting our theoretical results,  we briefly review a class of quasi-Newton updating rules for approximating a positive definite matrix $\mA \in \sR^{d\times d}$. We follow the definition by \cite{rodomanov2021greedy}, employing the following family of updates which describes the Broyden family \cite[Section 6.3]{nocedal2006numerical} of quasi-Newton updates, parameterized by a scalar $\tau\in\sR$. 

\begin{definition}
	Let $\mG \succeq \mA \succ \bm{0}$. For any $\vu \in\sR^d$, if $ \mG \vu = \mA\vu $, we define $ \Broyd_{\tau} (\mG, \mA, \vu):= \mG $. 
	Otherwise, i.e., $\mG\vu \neq \mA \vu$, we define
	\begin{equation}\label{eq:broyd}
	\begin{aligned}
	\Broyd_{\tau}(\mG,\mA,\vu) := \ & \tau \left[ \mG-\frac{\mA\vu\vu^\top \mG+\mG\vu\vu^\top \mA}{\vu^\top \mA\vu} +\left(\frac{\vu^\top \mG\vu}{\vu^\top \mA\vu}+1\right) \frac{\mA\vu\vu^\top \mA}{\vu^\top \mA\vu} \right] \\
	& + \left(1-\tau\right) \left[\mG-\frac{(\mG-\mA)\vu\vu^\top(\mG-\mA)}{\vu^\top(\mG-\mA)\vu}\right].
	\end{aligned}
	\end{equation}
\end{definition}

As mentioned in the work of \citet{rodomanov2021greedy}, we can recover several well-known quasi-Newton methods for several choices of $\tau$. 

For $\tau=0$, Eq.~\eqref{eq:broyd} corresponds to the well-known SR1 update:
\begin{equation}\label{eq:sr1}
    \mathrm{SR1}(\mG,\mA,\vu) := \mG-\frac{(\mG-\mA)\vu\vu^\top(\mG-\mA)}{\vu^\top(\mG-\mA)\vu}, \text{ if } \mG\vu \neq \mA \vu,
\end{equation}
and for $\tau=1$, it corresponds to the well-known DFP update:
\begin{equation}\label{eq:dfp}
    \DFP(\mG,\mA,\vu) := \mG-\frac{\mA\vu\vu^\top \mG+\mG\vu\vu^\top \mA}{\vu^\top \mA\vu} +\left(\frac{\vu^\top \mG\vu}{\vu^\top \mA\vu}+1\right) \frac{\mA\vu\vu^\top \mA}{\vu^\top \mA\vu}, \text{ if } \vu\neq \bm{0}.
\end{equation}
Finally, when $\tau = \frac{\vu^\top \mA\vu}{\vu^\top \mG\vu}\in [0,1]$, we recover the famous BFGS update\footnote{See Eq.~(2.6) in the work of \cite{rodomanov2021greedy} for derivation.}:
\begin{equation}\label{eq:bfgs}
    \BFGS(\mG,\mA,\vu) :=\mG-\frac{\mG\vu\vu^\top \mG}{\vu^\top \mG\vu}+\frac{\mA \vu\vu^\top \mA}{\vu^\top \mA\vu}, \text{ if } \vu\neq \bm{0}.
\end{equation}

The Broyden family has matrix monotonicity below, showing the relationship among these quasi-Newton methods. 
\begin{lemma} \emph{\citep[Lemmas 2.1 and  2.2]{rodomanov2021greedy}} \label{lemma:monotonic}
	If $\bm{0} \prec \mA \preceq \mG \preceq \eta \mA$ for some $\eta\geq 1$, then we have for any $\vu\in\sR^d$, and $\tau_1,\tau_2\in\sR$ with $\tau_1\leq \tau_2$ such that
	\[ \Broyd_{\tau_1}(\mG,\mA,\vu) \preceq \Broyd_{\tau_2}(\mG,\mA,\vu). \]
	And for any $\tau \in [0, 1]$, we have $ \mA \preceq \Broyd_{\tau}(\mG, \mA, \vu) \preceq \eta \mA $.
\end{lemma}

\subsection{Greedy and Random Quasi-Newton Updates}\label{sec:greedy}

\citet{rodomanov2021greedy} proposed a greedy version for selecting the direction $\vu$:
\begin{equation}\label{eq:gre-u-ori}
(\text{Greedy Broyden}) \ \ \hat{\vu}_{\mA}(\mG) := \mathop{\arg\max}_{\vu\in\{\ve_1, \dots, \ve_d\}} \frac{\vu^\top\mG\vu}{\vu^\top\mA\vu},
\end{equation}
which provides superlinear convergence of the form $\left(1-\frac{1}{d\varkappa}\right)^{k(k-1)/2}$.
They also conducted experiments to verify the performance of their greedy methods, which actually is competitive with the standard versions.
Moreover, they gave random quasi-Newton updates, that is,
\[ (\text{Random Broyden}) \ \ \vu \sim \gD \]
for some predefined distribution $\gD$. 
They observed that choosing a random direction uniformly from the standard Euclidean sphere, i.e., $\vu\sim\text{Unif}(\gS^{d-1})$, does not make superlinear convergence looser, and is only slightly slower than the greedy versions.
However, they did not provide the theory to support their experimental findings. 
We describe the distribution $\gD$ explicitly, and give a rigorous proof of the superlinear rates of such random methods in this paper.

\section{Rates of Random Quasi-Newton Methods}\label{sec:random}
We follow the same roadmap as the work of \citet{rodomanov2021greedy}. 
We begin with the analysis of quasi-Newton methods for approximating a target matrix. Then we extend the scheme to unconstrained quadratic minimization. Finally, we move to general strongly self-concordant functions.

\subsection{Matrix Approximation}\label{subsec:matrix-app-random}
We first consider approximating a positive definite matrix $\mA$ which satisfies 
\begin{equation}\label{eq:ass-A}
\mu \mI_d \preceq \mA \preceq L\mI_d, 
\end{equation}
where $L \geq \mu > 0$, and $\varkappa:=L/\mu$ is the condition number of $\mA$.
We use the measure $\sigma_{\mA}(\mG)$ to describe the closeness between matrix $\mA$ and the current approximate matrix $\mG$. 
When $\mG\vu \neq \mA\vu$, one iteration update of Broyden family leads to
\begin{eqnarray*}
\sigma_{\mA}(\mG_+) &\stackrel{\eqref{eq:sigmaA}\eqref{eq:broyd}}{=}& \tau \left[ \sigma_{\mA}(\mG) - 2 \cdot \frac{\vu^\top \mG\vu}{\vu^\top \mA\vu} +\left(\frac{\vu^\top \mG\vu}{\vu^\top \mA\vu}+1\right) \right] \\
&& + \left(1-\tau\right) \left[\sigma_{\mA}(\mG)-\frac{\vu^\top(\mG-\mA)\mA^{-1} (\mG-\mA)\vu}{\vu^\top(\mG-\mA)\vu}\right] \\
&=& \sigma_{\mA}(\mG) - \left[ \tau \cdot \frac{\vu^\top (\mG-\mA)\vu}{\vu^\top \mA\vu} + \left(1-\tau\right) \cdot \frac{\vu^\top(\mG-\mA)\mA^{-1} (\mG-\mA)\vu}{\vu^\top(\mG-\mA)\vu}\right],
\end{eqnarray*}
where $\mG_+ = \text{Broyd}_{\tau}(\mG, \mA, \vu)$.
Note that we always have $\mG_+ \succeq \mA$ for $\tau\in[0, 1]$ if $\mG\succeq \mA$ from Lemma \ref{lemma:monotonic}.
Thus, by the Cauchy–Schwarz inequality and $\mG\succeq \mA$, we have
\[ \frac{\vu^\top(\mG-\mA)\mA^{-1} (\mG-\mA)\vu}{\vu^\top(\mG-\mA)\vu} \geq \frac{\vu^\top(\mG-\mA)\vu}{\vu^\top\mA\vu}. \]
Hence, we obtain when $\mG\vu \neq \mA\vu$,
\begin{align}\label{eq:sigma-bound}
    \sigma_{\mA}(\mG_+) &\leq \sigma_{\mA}(\mG) - \frac{\vu^\top (\mG-\mA)\vu}{\vu^\top \mA\vu} \stackrel{\eqref{eq:ass-A}}{\leq} \sigma_{\mA}(\mG) - \frac{1}{L} \cdot \frac{\vu^\top (\mG-\mA)\vu}{\vu^\top\vu}  \nonumber \\
    &= \sigma_{\mA}(\mG) - \frac{1}{L} \tr\left[ (\mG-\mA) \cdot \frac{\vu\vu^\top }{\vu^\top\vu} \right].
\end{align}
Moreover, Eq.~\eqref{eq:sigma-bound} trivially holds when $\mG\vu = \mA\vu$.
Therefore, for a random direction $\vu$, we only need $\E \vu\vu^\top/\vu^\top\vu$ to preserve some benign property, which leads to our assumption of the random update distribution.
\begin{equation}\label{eq:randomu}
(\text{Random Broyden}) \ \ \vu\sim\gD, \; s.t. \; \E_{\vu\sim\gD} \, \frac{\vu\vu^\top}{\vu^\top\vu} = \frac{1}{d}\mI_d.
\end{equation}
It is easy to verify that common distributions such as $\mathcal{N}(\bm{0}, \mI_d)$ and $\mathrm{Unif}(\mathcal{S}^{d-1})$ satisfy our requirements. 
Based on Eq.~\eqref{eq:randomu} and update in Algorithm \ref{algo:random-update}, we could show linear convergence of $\mG_k$ to $\mA$ under measure $\sigma_{\mA}(\cdot)$. The proof of Theorem \ref{thm:rand-update} is shown in Appendix \ref{app:miss-prove-random1}.

\begin{algorithm}[t]
	\caption{Random quasi-Newton updates}
	\begin{algorithmic}
		\STATE Initialization: Choose $\mG_0 \succeq \mA$.
		\FOR{$ k \geq 0 $}
		\STATE Choose $\tau_k \in [0, 1]$ and $\vu_k$ from distribution $\gD$ which satisfies Eq.~\eqref{eq:randomu}.
		\STATE Compute $\mG_{k+1} = \text{Broyd}_{\tau_k}(\mG_k, \mA, \vu_k)$.
		\ENDFOR
	\end{algorithmic}
	\label{algo:random-update}
\end{algorithm}

\begin{theorem}\label{thm:rand-update}
	Under the update in Algorithm \ref{algo:random-update} with	a randomly initialized $\mG_0$, such that $\mG_0 \succeq \mA$ always holds,  we have that
	\begin{equation}\label{eq:random-k}
	\forall k\geq 0, \; \mG_k\succeq \mA \; \text{ and } \; 0 \leq \E \sigma_k \leq \left(1-\frac{1}{d\varkappa}\right)^{k} \E \sigma_0, 
	\end{equation}
Therefore, $\E \sigma_{\mA}(\mG_k)$ converges to zero linearly.
\end{theorem}

\subsection{Unconstrained Quadratic Minimization}\label{subsec:quad-random}
\begin{algorithm}[t]
	\caption{Random quasi-Newton methods for quadratic minimization}
	\begin{algorithmic}[1]
		\STATE Initialization: Choose $\vx_0\in\sR^d$ and $\mG_0 \succeq \mA$.
		\FOR{$ k \geq 0 $}
		\STATE Update $\vx_{k+1} = \vx_{k} - \mG_k^{-1} \nabla f(\vx_k)$.
		\STATE Choose $\tau_k \in [0, 1]$ and $\vu_k$ from distribution $\gD$ which satisfies Eq.~\eqref{eq:randomu}.
		\STATE Compute $\mG_{k+1} = \text{Broyd}_{\tau_k}(\mG_k, \mA, \vu_k)$.
		\ENDFOR
	\end{algorithmic}
	\label{algo:quad-random}
\end{algorithm}

Based on the efficiency of random quasi-Newton updates in matrix approximation, we next turn to minimize the strongly convex quadratic function (with a fixed Hessian):
\begin{equation}\label{eq:object-f}
f(\vx) = \frac{1}{2} \vx^\top \mA\vx-\vb^\top\vx, \text{ where } \mu \mI_d \preceq \mA \preceq L\mI_d \text{ with } L, \mu > 0.
\end{equation}
The algorithm is shown in Algorithm \ref{algo:quad-random}. 
As classical quasi-Newton methods do, we need to use the quasi-Newton step for updating the parameters as well as approximating the true Hessian matrix $\mA$. 
Moreover, Algorithm \ref{algo:quad-random} is only for theoretical analysis, while we need to adopt the inverse update rules for $\mG_k^{-1}$ directly in practice.

We adopt $\{\lambda_k\}$ (defined in Eqs.~\eqref{eq:lambda} and \eqref{eq:lam-sig}) to estimate the convergence rate of the objective in Eq.~(\ref{eq:object-f}).
Note that this measure of optimality is directly related to the functional residual.
Indeed, note that $\vx_* = \mA^{-1}\vb $ is the minimizer of Eq.~\eqref{eq:object-f}. Then we obtain
\[f(\vx)-f(\vx_*) = \frac{1}{2}\left(\vx-\vx_*\right)^\top \mA\left(\vx-\vx_*\right) = \frac{1}{2}\left(\mA\vx-\vb\right)^\top \mA^{-1}\left(\mA\vx-\vb\right) \stackrel{\eqref{eq:lambda}}{=} \frac{1}{2}\lambda_f(\vx)^2.  \]
The following lemma shows how $\lambda_f(\cdot)$ changes after one iteration of process in Algorithm \ref{algo:quad-random}.

\begin{lemma} \emph{\citep[Lemma 3.2]{rodomanov2021greedy}} \label{lemma:lambda}
	Let $k\geq 0$, and $\eta_k\geq 1$ be such that $\mA \preceq \mG_k\preceq\eta_k \mA$. Then we have
	$ \lambda_{k+1} \leq \left(1-\frac{1}{\eta_k}\right)\lambda_k \leq (\eta_k-1)\lambda_k$. 
\end{lemma}
Thus, to estimate how fast $\{\lambda_k\}$ converges to zero, we need the upper bound $\eta_k$, which was already done in Theorem \ref{thm:rand-update}.
Therefore, we can guarantee a superlinear convergence of $\{\lambda_k\}$ (under expectation) using the random quasi-Newton update. The proof of Theorem \ref{thm:random-quad} can be found in Appendix \ref{app:miss-prove-random2}.

\begin{theorem}\label{thm:random-quad}
	Under the update in Algorithm \ref{algo:quad-random} with a randomly initialized $\mG_0$, such that $\mG_0 \succeq \mA$ always holds, we have that $\forall k \geq 0, \lambda_{k+1} \leq \rho_k \lambda_k$, where $\rho_k$ is a certain nonnegative random variable such that
	\[ \E \rho_k \leq \left(1-\frac{1}{d\varkappa}\right)^k \E\sigma_0, \;  \forall k \geq 0. \]
\end{theorem}

For better understanding the convergent behavior without expectation, we show the probabilistic version of Theorems \ref{thm:rand-update} and \ref{thm:random-quad} below, and leave the proof in Appendix \ref{app:miss-prove-random2.5}.

\begin{corollary}\label{cor:quad-ranom-p}
    Under the same assumptions as Theorem \ref{thm:random-quad}, for any $\delta \in (0, 1)$, with probability at least $1-\delta$ over the random directions $\{\vu_k\}$, we have for all $k \geq 0$,
    \[ \sigma_k \leq \frac{2d^2\varkappa^2 \E\sigma_0}{\delta} \left(1-\frac{1}{d\varkappa+1}\right)^{k}  \text{ and } \lambda_k \leq \left(\frac{2d^2\varkappa^2 \E\sigma_0}{\delta}\right)^k  \left(1-\frac{1}{d\varkappa+1}\right)^{k(k-1)/2} \lambda_0. \]
\end{corollary}

\subsection{Minimization of General Functions}\label{subsec:random-gel}
Next, we consider the optimization of a general machine learning objective: $\min_{\vx\in\sR^d} f(\vx)$, where $ f \colon \sR^d \to \sR $ is an $M$-strongly self-concordant, $\mu$-strongly convex and $L$-smooth function with condition number $\varkappa=L/\mu$.
Our goal is to extend the results in the previous sections, given that the methods can start from a sufficiently good initial point $\vx_0$.
Unlike quadratic minimization, the true Hessian in each step varies. 
In order to ensure that $\mG_{k+1} \succeq \nabla^2 f(\vx_{k+1})$ holds for all $k \geq 0$, we adjust $\mG_k$ before doing quasi-Newton update. 
Instructed from the work of \cite{rodomanov2021greedy}, we also use the \textit{correction strategy}, which enlarges the approximation $\mG_k$ properly shown in Line 4 of Algorithm \ref{algo:general-random}. 
Note that Algorithm \ref{algo:general-random} is only for theoretical analysis. We will use the inverse update rules for $\mG_k^{-1}$ and Hessian-vector products for $\text{Broyd}_{\tau_k}(\tilde{\mG}_k, \nabla^2 f(\vx_{k+1}), \vu_k)$ in practice.
For simplicity, we assume that the constants $ M $ and $ L $ are available, and $d \geq 2$.
We first give convergent results in expectation in Lemma \ref{lemma:random-gel}, and leave the
proof in Appendix \ref{app:miss-prove-random3}.

\begin{algorithm}[t]
	\caption{Random quasi-Newton methods for general strongly self-concordant objective}
	\begin{algorithmic}[1]
		\STATE Initialization: Choose $\vx_0\in\sR^d$ and $\mG_0 \succeq \nabla^2 f(\vx_0)$.
		\FOR{$ k \geq 0 $}
		\STATE Update $\vx_{k+1} = \vx_{k} - \mG_k^{-1} \nabla f(\vx_k)$.
		\STATE Compute $r_k = \|\vx_{k+1}-\vx_k\|_{\vx_k}$ and set $\tilde{\mG}_k = \left(1+Mr_k\right)\mG_k$.
		\STATE Choose $\tau_k \in [0, 1]$ and $\vu_k$ from distribution $\gD$ which satisfies Eq.~\eqref{eq:randomu}.
		\STATE Compute $\mG_{k+1} = \text{Broyd}_{\tau_k}(\tilde{\mG}_k, \nabla^2 f(\vx_{k+1}), \vu_k)$.
		\ENDFOR
	\end{algorithmic}
	\label{algo:general-random}
\end{algorithm}

\begin{lemma}\label{lemma:random-gel}
	Suppose in Algorithm \ref{algo:general-random}, a random initialization $\mG_0$ always satisfies $\nabla^2 f(\vx_0) \preceq \mG_0 \preceq \eta \nabla^2 f(\vx_0)$ for some $\eta \geq 1$, and the initial point $\vx_0$ is sufficiently close to the solution:
	\begin{equation*}
	    M\lambda_0 \leq \frac{\ln2}{4\eta(2d+1)}.
	\end{equation*}
	Then for all $k \geq 0$, we have $\nabla^2 f(\vx_k) \preceq \mG_k \preceq (1+\delta_k)\nabla^2 f(\vx_k)$, where $\delta_k$ is a certain nonnegative random variable such that
	\begin{equation*}
	    \E \delta_k \leq 2d\eta \left(1-\frac{1}{d\varkappa}\right)^k,
	\end{equation*}
	and $\lambda_{k+1} \leq \rho_k\lambda_k$, where $\rho_k$ is a certain nonnegative random variable such that
	\begin{equation*}
	    \E \rho_k \leq 2d \eta \left(1-\frac{1}{d\varkappa}\right)^k.
	\end{equation*}
\end{lemma}

We also show the probabilistic version of Lemma \ref{lemma:random-gel}, which gives superlinear convergence of $\{\lambda_k\}$ and linear convergence of $\{\delta_k\}$ directly, and we leave the proof in Appendix \ref{app:miss-prove-random4}.

\begin{theorem}\label{thm:random-p}
	Under the same assumptions and notation as in Lemma \ref{lemma:random-gel}, for any $\delta \in (0, 1)$, with probability at least $1-\delta$ over the random directions $\{\vu_k\}$, we have for all $k \geq 0$,
	\begin{equation*}
	\delta_k \leq \frac{4d^3\varkappa^2\eta}{\delta} \left(1 - \frac{1}{d\varkappa+1}\right)^{k} \; \text{ and } \; \lambda_{k} \leq \left(\frac{4d^3\varkappa^2\eta}{\delta}\right)^{k} \left(1 - \frac{1}{d\varkappa+1}\right)^{k(k-1)/2} \lambda_0.
	\end{equation*}
\end{theorem}

Additionally, as mentioned by \citet{rodomanov2021greedy}, if we adopt a weaker initialization of $\vx_0$, then the superlinear rate is valid only after certain iterations, i.e., the total iteration count $k' \geq k_0$ for some $k_0$, while only linear convergence is guaranteed for $k' < k_0$. 
We combine both phases into the following corollary, and leave the proof in Appendix \ref{app:miss-prove-random5}.
\begin{corollary}\label{cor:random-gel}
	Suppose in Algorithm \ref{algo:general-random}, $\mG_0=L \mI_d$ and $\vx_0$ satisfies $M\lambda_0 \leq \frac{\ln\frac{3}{2}}{4\varkappa}$, that is, the initial condition on $\lambda_0$ here is weaker than that in Lemma \ref{lemma:random-gel} with $\eta=\varkappa$. Then we could obtain with probability at least $1-\delta$ over the random directions $\{\vu_k\}$, 
	\[ \lambda_{k_0+k} \leq \left(1-\frac{1}{d\varkappa+1}\right)^{k(k-1)/2} \cdot \left(\frac{1}{2}\right)^k \cdot \left(1-\frac{1}{2\varkappa}\right)^{k_0} \cdot \lambda_0, \forall k \geq 0, \]
	where $k_0 = O\left(d\varkappa\ln(d\varkappa/\delta)\right)$.
\end{corollary}

Hence, we could see random quasi-Newton methods still have explicit superlinear convergence rates. The rate in Corollary \ref{cor:random-gel} is slightly worse than the bounds in greedy methods (see Table \ref{table:res}) due to the probabilistic version, but is comparable overall.

\section{Faster Rates for the BFGS and SR1 Methods}\label{sec:faster}

From Lemma \ref{lemma:monotonic}, if $\mA \preceq \mG \preceq \eta \mA$ for some $\eta \geq 1$, it follows that
\begin{equation*}
	\mA \preceq \text{SR1}(\mG,\mA,\vu) \preceq \text{BFGS}(\mG,\mA,\vu) \preceq \text{DFP}(\mG,\mA,\vu) \preceq \eta \mA.
\end{equation*}
Intuitively, the approximation produced by SR1 is better than that produced by BFGS. And both of them are better than that produced by DFP. 
However, \citet{rodomanov2021greedy} reduced the analysis by casting all updates described by Broyden family ($\tau \in [0, 1]$) into the slowest DFP update ($\tau=1$).
Moreover, SR1 and BFGS methods also have faster numerical performance in practice.
Therefore, \citet{rodomanov2021greedy} conjectured that SR1 and BFGS methods might have faster superlinear convergence rates. 
In this section, we will provide an affirmative answer to this conjecture.

\subsection{Superlinear Convergence for SR1 Update}

We first describe the SR1 update for approximating a fixed positive definite matrix $\mA \in \sR^{d\times d}$. 
Let us now justify the efficiency of update Eq.~\eqref{eq:sr1} in ensuring convergence $\mG$ to $\mA$.
We adopt another measure $\tau_{\mA}(\cdot)$ instead of $\sigma_{\mA}(\cdot)$. According to $\tau_{\mA}(\mG)$, one iteration update leads to
\begin{equation}\label{eq:tau_update}
\tau_{\mA}(\mG_+) \stackrel{\eqref{eq:tauA}\eqref{eq:sr1}}{=} \tau_{\mA}(\mG) - \frac{\vu^\top(\mG-\mA)^2\vu}{\vu^\top(\mG-\mA)\vu}, \; \; \ \mG_+ = \mathrm{SR1}(\mG, \mA, \vu), \text{ if } \mG\vu \neq \mA\vu.
\end{equation}
Now we revise greedy and random methods based on the progress of measure $\tau_{\mA}(\cdot)$.

First, we introduce greedy method proposed in the work of  \cite{rodomanov2021greedy}, that greedily selects $\vu$ from the basis vectors to obtain the largest decrease of $\tau_{\mA}(\mG_+)-\tau_{\mA}(\mG)$:
\[ \bar{\vu}_{\mA}^{raw}(\mG) := \mathop{\arg\max}_{\vu\in\{\bm{e}_1,\dots, \bm{e}_d\}}\frac{\vu^\top(\mG-\mA)^2\vu}{\vu^\top(\mG-\mA)\vu}. \]

However, we may encounter numerical overflow due to division by zero if $\vu^\top(\mG-\mA)\vu$ is nearly $0$. 
Noting that $\mG \succeq \mA$, then from the Cauchy–Schwarz inequality, we have 
\begin{equation}\label{eq:cauchy}
    \frac{\vu^\top(\mG-\mA)^2\vu}{\vu^\top(\mG-\mA)\vu} \geq \frac{\vu^\top(\mG-\mA)\vu}{\vu^\top\vu}.
\end{equation}
Thus we employ a safer adjustment below:
\begin{equation}\label{eq:greedy-sr1}
(\text{Greedy SR1}) \ \ \bar{\vu}_{\mA}(\mG) := \mathop{\arg\max}_{\vu \in \{\bm{e}_1, \ldots, \bm{e}_d\}} \frac{\vu^\top(\mG-\mA)\vu}{\vu^\top\vu} = \mathop{\arg\max}_{\vu\in\{\bm{e}_1,\ldots, \bm{e}_d\}}\vu^\top(\mG-\mA)\vu.
\end{equation}
Moreover, we only need to obtain the diagonal elements of $\mA$ (the current Hessian in practice), thus generally the total complexity is $O(d^2)$ in each iteration\footnote{Note that we can use the Hessian-vector product to obtain $\mA\vu$ (or $\nabla^2 f(\vx)\cdot \vu$) in practice. For most specific optimization problems, e.g., two problems in our experiments, one operation of the exact Hessian-vector product is tractable with $O(d)$ complexity.}, which is acceptable and the same as the classical quasi-Newton methods.

Second, from the proof of the greedy method, we find that the random method by
choosing $\vu$ from a spherically symmetric distribution, e.g.,
\begin{equation}\label{eq:random-sr1}
(\text{Random SR1}) \ \ \vu\sim\mathcal{N}(0, \mI_d) \; \text{ or } \; \vu \sim \mathrm{Unif}(\mathcal{S}^{d-1}),
\end{equation}
also has similar performance and the same running complexity $O(d^2)$ in each iteration.

Next, we will show the convergence result below by estimating the decrease in the measure $\tau_{\mA}(\cdot)$. 
In the following, the expectation considers all the randomness of the directions $\{\vu_k\}$ during iterations, and when applied to the greedy method, we can view it with no randomness for the same notation. We leave the proof of Theorem \ref{thm:sr1-update} in Appendix \ref{app:miss-prove1}. 

\begin{figure}[tp]
	\vspace*{-\baselineskip}
	\begin{minipage}[t]{.49\textwidth} 
		\begin{algorithm}[H]
			\caption{Greedy/Random $\mathrm{SR1}$ update}
			\begin{algorithmic}[1]
				\STATE Initialization: Choose $\mG_0 \succeq \mA$.
				\FOR{$ k=0, \dots, d-1 $}
				\STATE Choose $\vu_k$ from \\
				1) \textit{greedy method}: $\vu_k = \bar{\vu}_{\mA}(\mG_k)$, or \\
				2) \textit{random method}: $\vu_k \sim \mathrm{Unif}(\mathcal{S}^{d-1})$.
				\STATE Compute $\mG_{k+1} = \mathrm{SR1}(\mG_k, \mA, \vu_k)$. \\
				\quad 
				\ENDFOR
			\end{algorithmic}
			\label{algo:sr1-update}
		\end{algorithm}
	\end{minipage}
	\begin{minipage}[t]{.5\textwidth}
		\begin{algorithm}[H]
			\caption{Greedy/Random $\BFGS$ update}
			\begin{algorithmic}[1]
				\vspace{-3pt}
				\STATE Initialization: Set $\mG_0 \succeq \mA$, $\mL_0^\top\mL_0 {=} \mG_0^{-1}$.
				\FOR{$ k \geq 0 $}
				\STATE Compute $\vu_k = \bm{L}_k^\top\tilde{\vu}_k$ with $\tilde{\vu}_k$ from \\
				1) \textit{greedy method}: $\tilde{\vu}_k = \tilde{\vu}_{\mA}(\bm{L}_k)$, or \\
				2) \textit{random method}: $\tilde{\vu}_k \sim \mathrm{Unif}(\mathcal{S}^{d-1})$.
				\STATE Compute $\mG_{k+1} = \mathrm{BFGS}(\mG_k, \mA, \vu_k)$.
				\STATE Compute $\bm{L}_{k+1}$ based on Eq.~\eqref{eq:lk-up}.
				\ENDFOR
			\end{algorithmic}
			\label{algo:bfgs-update}
		\end{algorithm}
	\end{minipage}
\end{figure}

\begin{theorem}\label{thm:sr1-update}
	Suppose in Algorithm \ref{algo:sr1-update}, a random initialization $\mG_0$ always satisfies $\mG_0 \succeq \mA$. Then we obtain that for the greedy method defined in Eq.~\eqref{eq:greedy-sr1} or the random method defined in Eq.~\eqref{eq:random-sr1},
	\begin{equation}\label{eq:tau-k}
	    \forall k \geq 0, \mG_k \succeq \mA \text{ and } 0 \leq \E \tau_k \leq \left(1-\frac{k}{d}\right)_+ \E\tau_0,
	\end{equation}
	where $(x)_+=\max\{0,x\}$. Hence, $\E \tau_{\mA}(\mG_k)$ converges to zero superlinearly.
	Particularly, $\forall k\geq d, \mG_k = \mA$ for greedy SR1 update, and $\mG_k = \mA$ almost surely for random SR1 update.
\end{theorem}

Previous work \citep{rodomanov2021greedy} adopted measure $\sigma_{\mA}(\cdot)$, which only gives the same rate as general updates of the Broyden family  \citep[e.g., Theorem \ref{thm:rand-update}, and][Theorem 2.5]{rodomanov2021greedy}. 
Thus we employ a more precise measure $\tau_{\mA}(\cdot)$.

\subsection{Linear Convergence for BFGS Update}
We now consider the classical BFGS update in the same scheme. Reusing the measure $\sigma_{\mA}(\cdot)$, we obtain that
\begin{equation}\label{eq:sigma-up}
\sigma_{\mA}(\mG_+) \stackrel{\eqref{eq:sigmaA}\eqref{eq:bfgs}}{=} \sigma_{\mA}(\mG) - \frac{\vu^\top \mG \mA^{-1}\mG\vu}{\vu^\top \mG \vu} + 1, \; \; \mG_+ = \BFGS(\mG, \mA, \vu), \text{ if } \vu \neq \bm{0}.
\end{equation}
If we directly apply the greedy or random method from the previous content, we could only obtain the same linear convergence rate as \citet[Theorem 2.5]{rodomanov2021greedy}.
However, if we take advantage of the current $\mG$, and choose a scaled direction such that $\vu = \mL^\top\tilde{\vu}$ where $\mL$ is a square matrix satisfying $\mL^\top\mL=\mG^{-1}$, then we could simplify the formulation and obtain a faster condition-number-free linear convergence rate.
Specifically, after replacing $\vu$ with $\mL^\top\tilde{\vu}$ and $\mG$ with $\mL^{-1}\mL^{-\top}$, we get
\begin{equation}\label{eq:sigma-up2}
    \sigma_{\mA}(\mG_+) \stackrel{\eqref{eq:sigma-up}}{=} \sigma_{\mA}(\mG) - \frac{\tilde{\vu}^\top \bm{L}^{-\top} \mA^{-1}\bm{L}^{-1}\tilde{\vu}}{\tilde{\vu}^\top \tilde{\vu}} + 1, \; \; \mG_+ = \BFGS(\mG, \mA, \vu), \text{ if } \vu \neq \bm{0}. 
\end{equation}
Thus our modified greedy BFGS update is as follows:
\begin{equation}\label{eq:greedy-bfgs}
    (\text{Greedy BFGS}) \ \ \tilde{\vu}_{\mA}(\mL) = \mathop{\arg\max}_{\tilde{\vu} \in \{\bm{e}_{1}, \dots, \bm{e}_d\}} \tilde{\vu}^\top \mL^{-\top}\mA^{-1}\mL^{-1}\tilde{\vu}.
\end{equation}
Similar arguments apply to the random method used in Eq.~\eqref{eq:randomu}:
\begin{equation}\label{eq:random-bfgs}
    (\text{Random BFGS}) \ \ \tilde\vu \sim\gD, \ s.t. \; \E_{\tilde{\vu} \sim \gD} \, \frac{\tilde{\vu}\tilde{\vu}^\top}{\tilde{\vu}^\top\tilde{\vu}} = \frac{1}{d} \mI_d.
\end{equation}
Now we give the linear convergence rate of the BFGS update under our modified method. We leave the proof of Theorem \ref{thm:bfgs-update} in Appendix \ref{app:miss-prove2}.

\begin{theorem}\label{thm:bfgs-update}
    Suppose in Algorithm \ref{algo:bfgs-update}, a random initialization $\mG_0$ always satisfies $\mG_0 \succeq \mA$. Then we obtain that for the \textit{greedy method} defined in Eq.~\eqref{eq:greedy-bfgs} or the \textit{random method} defined in Eq.~\eqref{eq:random-bfgs},
    \begin{equation} \label{eq:sigma-k}
        \forall k\geq 0, \mG_k\succeq \mA \text{ and } 0 \leq \E \sigma_k \leq \left(1-\frac{1}{d}\right)^k \E \sigma_0.
    \end{equation}
    Therefore, $ \E \sigma_A(\mG_k)$ converges to zero linearly.
\end{theorem}

\begin{remark}\label{remark:bfgs-eff}
	Note that the complexity in Eq.~\eqref{eq:greedy-bfgs} is $O(d^3)$ because we have  multiplication-addition operations with (unknown) $\mA^{-1}$. Hence we \textbf{do not} apply this greedy strategy in practice, but view it as a theoretical result similar to the random strategy.  
	Moreover, the random method is still practical, and we will show the efficiency of our scaled direction compared to the original direction in our numerical experiments.
\end{remark}

Finally, we can employ an efficient way (with complexity $O(d^2)$) for updating $\mL_k$ at each step $k\geq 0$, and we leave the proof in Appendix \ref{app:propl}.

\begin{proposition}\label{prop:l}
    Suppose we already have $\mL_k^\top \mL_k = \mG_k^{-1} \succ \bm{0}$, where $\mL_k$ is a square matrix, and $\vu_k=\mL_k^\top\tilde\vu_k$. 
    Then we can construct the square matrix $\mL_{k+1}$ which satisfies $\mL_{k+1}^\top \mL_{k+1}= \mG_{k+1}^{-1} := [\BFGS(\mG_k, \mA,\vu_k)]^{-1}$ as below:
    \begin{equation}\label{eq:lk-up}
        \mL_{k+1} = \mL_k - \frac{\left[\mL_k (\mA\vu_k) - \vv_k\right] \vu_k^\top}{\vu_k^\top (\mA \vu_k)} \text{ with } \vv_k = \sqrt{\vu_k^\top \cdot (\mA\vu_k)} \cdot \frac{\tilde\vu_k}{\norm{\tilde\vu_k}}.
    \end{equation}
\end{proposition}

\subsection{Unconstrained Quadratic Minimization}\label{sec:un-quad}

\begin{algorithm}[t]
	\caption{Greedy/Random SR1/BFGS methods for quadratic minimization}
	\begin{algorithmic}[1]
		\STATE Initialization: Choose $\vx_0\in\sR^d$ and $\mG_0 \succeq \mA, \mL_0^\top\mL_0 = \mG_0^{-1}$.
		\FOR{$ k \geq 0 $}
		\STATE Update $\vx_{k+1} = \vx_{k} - \mG_k^{-1} \nabla f(\vx_k)$. Choose one of the following update rules:
		\STATE (\romannumeral1) SR1: Choose $\vu_k$ following Algorithm \ref{algo:sr1-update}. Compute $\mG_{k+1} = \mathrm{SR1}(\mG_k, \mA, \vu_k)$.
		\STATE (\romannumeral2) BFGS: Choose $\vu_k$ following Algorithm \ref{algo:bfgs-update}. Compute $\mG_{k+1} = \BFGS(\mG_k, \mA, \vu_k)$.
		\ENDFOR
	\end{algorithmic}
	\label{algo:quad-quasi}
\end{algorithm}

Based on the efficiency of the greedy/random SR1 and BFGS updates in matrix approximation, we next turn to minimize the strongly convex quadratic function in Eq.~\eqref{eq:object-f}.
We show the detail in Algorithm \ref{algo:quad-quasi}, which is only for theoretical analysis. 
In practice, we use the inverse update rules \citep[Eqs. (6.17) and (6.25)]{nocedal2006numerical} to update $\mG_k^{-1}$:
\begin{align}
\mG_{+}^{-1} &= \mG^{-1} + \frac{(\mI_d-\mG^{-1}\mA)\vu\vu^\top(\mI_d-\mA\mG^{-1})}{\vu^\top(\mA-\mA\mG^{-1}\mA)\vu}, && \mG_+ = \mathrm{SR1}(\mG, \mA, \vu); \label{eq:invsr1} \\
\mG_{+}^{-1} &= \left(\mI_d-\dfrac{\vu\vu^\top \mA}{\vu^\top \mA\vu}\right) \mG^{-1} \left(\mI_d-\dfrac{\mA \vu\vu^\top}{\vu^\top \mA\vu}\right) + \dfrac{\vu\vu^\top}{\vu^\top \mA \vu}, && \mG_+ = \mathrm{BFGS}(\mG, \mA, \vu).  \label{eq:invbfgs}
\end{align}
Based on Lemma \ref{lemma:lambda}, we can guarantee a faster superlinear convergence of $\{\lambda_k\}$ (defined in Eqs.~\eqref{eq:lambda} and \eqref{eq:lam-sig}) using the greedy\slash random SR1 or BFGS update. 
The proof of Theorem \ref{thm:quad} can be found in Appendix \ref{app:miss-prove-random2}.

\begin{theorem}\label{thm:quad}
	For Algorithm \ref{algo:quad-quasi} with a randomly initialized $\mG_0$, such that $\mG_0 \succeq \mA$ always holds, we have that $\forall k \geq 0, \lambda_{k+1} \leq \rho_k\lambda_{k}$, where $\rho_k$ is a certain nonnegative random variable such that for SR1 update,
	\[ \E \rho_k \leq \left(1-\frac{k}{d}\right)_+\frac{\E\tau_0}{\mu}, \; \forall k \geq 0,  \]
	and for BFGS update,
	\[ \E \rho_k \leq \left(1-\frac{1}{d}\right)^k \E\sigma_0, \; \forall k \geq 0. \] 
\end{theorem}

We can also use a similar technique in Corollary \ref{cor:quad-ranom-p} to give the probabilistic version of Theorem \ref{thm:quad}, but the differences from greedy/random quasi-Newton methods are clear.
In particular, for the SR1 update, our bound recovers the classical result of \citet[Theorem 6.1]{nocedal2006numerical}, showing that the update stops after finite steps because $\mG_d=\mA$ and $\lambda_{d+1}=0$ almost surely. 
Moreover, we give an explicit rate during the entire optimization process.
And the main decreasing term $\left(1-\frac{k}{d}\right)_+$ for the SR1 update as well as $(1-\frac{1}{d})^k$ for the BFGS update in the $k$-th iteration are independent of the condition number $\varkappa$ of $\mA$, which improves the bound $(1-\frac{1}{d\varkappa})^k$ by \citet[Theorem 3.4]{rodomanov2021greedy}.

\subsection{Minimization of General Functions}\label{sec:un-general}

Finally, we consider the optimization of an $M$-\textit{strongly self-concordant}, $\mu$-strongly convex and $L$-smooth objective as Subsection \ref{subsec:random-gel} does.
We show the entire iteration coupled with our modified update rules in Algorithm \ref{algo:general-quasi}.
We underline that Algorithm \ref{algo:general-quasi} is only for theoretical analysis, and we will use the inverse update rules (Eqs.~\eqref{eq:invsr1} and \eqref{eq:invbfgs}) and Hessian-vector products in practice.
Additionally, we assume that $d \geq 2$, and the constants $ M $ and $ L $ are available for simplicity.
Using the same proof technique, we could obtain faster convergence rates of $\{\lambda_k\}$ (defined in Eqs.~\eqref{eq:lambda} and \eqref{eq:lam-sig}) for greedy/random SR1 or BFGS method. 
The proof of Lemma \ref{lemma:gel-bfgs-sr1} can be found in Appendix \ref{app:miss-prove-random3}.

\begin{algorithm}[t]
	\caption{Greedy/Random SR1/BFGS methods for strongly self-concordant objective}
	\begin{algorithmic}[1]
		\STATE Initialization: Choose $\vx_0\in\sR^d$ and $\mG_0 \succeq \nabla^2 f(\vx_0), \mL_0^\top\mL_0 = \mG_0^{-1}$.
		\FOR{$ k \geq 0 $}
		\STATE Update $\vx_{k+1} = \vx_{k} - \mG_k^{-1} \nabla f(\vx_k)$.
		\STATE Compute $r_k = \|\vx_{k+1}-\vx_k\|_{\vx_k}$, $\tilde{\mG}_k = \left(1+Mr_k\right)\mG_k$, $\tilde{\mL}_k = \mL_k / \sqrt{1+Mr_k}$.
		\STATE (\romannumeral1) Greedy/Random SR1: Choose $\vu_k = \bar{\vu}_{\nabla^2f(\vx_{k+1})}(\tilde{\mG}_k)$, or $\vu_k \sim \mathrm{Unif}(\mathcal{S}^{d-1})$. \\
		Compute $\mG_{k+1} = \mathrm{SR1}(\tilde{\mG}_k, \nabla^2 f(\vx_{k+1}), \vu_k)$.
		\STATE (\romannumeral2) Greedy/Random BFGS: Choose $\vu_k = \tilde{\bm{L}}_k^\top\tilde{\vu}_k$ with $\tilde{\vu}_k = \tilde{\vu}_{\nabla^2f(\vx_{k+1})}(\tilde{\bm{L}}_k)$, or $\tilde{\vu}_k \sim \mathrm{Unif}(\mathcal{S}^{d-1})$. \\
		Compute $\mG_{k+1} = \BFGS(\tilde{\mG}_k, \nabla^2f(\vx_{k+1}), \vu_k)$, and $\bm{L}_{k+1}$ based on Eq.~\eqref{eq:lk-up}.
		\ENDFOR
	\end{algorithmic}
	\label{algo:general-quasi}
\end{algorithm}

\begin{lemma}\label{lemma:gel-bfgs-sr1}
	Suppose in Algorithm \ref{algo:general-quasi}, a randomly initialized $\mG_0$ always satisfies $\nabla^2 f(\vx_0) \preceq \mG_0 \preceq \eta \nabla^2 f(\vx_0)$ for some $\eta \geq 1$, and the initial point $\vx_0$ is sufficiently close to the solution:
	\begin{equation*}
	    M\lambda_0 \leq \frac{\ln2}{4\eta(2c d +1)},
	\end{equation*}
	where $c=1$ for BFGS update and $c=\varkappa$ for SR1 update.
	Then for all $k \geq 0$, we have $\nabla^2 f(\vx_k) \preceq \mG_k \preceq (1+\delta_k)\nabla^2 f(\vx_k)$, where $\delta_k$ is a certain nonnegative random variable such that
	\begin{equation*}
	    \E \delta_k \leq 2c d\eta \left(1-\frac{1}{d}\right)^k,
	\end{equation*}
	and $\lambda_{k+1} \leq \rho_k\lambda_k$, where $\rho_k$ is a certain nonnegative random variable such that
	\begin{equation*}
	    \E \rho_k \leq 2c d\eta\left(1-\frac{1}{d}\right)^k.
	\end{equation*}
\end{lemma}

Similarly, we can give deterministic results of greedy methods and probabilistic results of randomized methods below. We leave the proof of Theorem \ref{thm:gre-ran-p} in Appendix \ref{app:miss-prove-random4}.

\begin{theorem}\label{thm:gre-ran-p}
    Under the same assumptions and notation as in Lemma \ref{lemma:gel-bfgs-sr1}, we have the explicit rates of $\{\lambda_k\}$ and $\{\delta_k\}$ shown in below:
	\begin{itemize}
	    \item for greedy BFGS/SR1 method, we have
        	\begin{equation*}
        		\delta_k \leq 2c d \eta\left(1-\frac{1}{d}\right)^{k} \text{  and  }
        		\lambda_k \leq \left(2c d\eta\right)^k \left(1-\frac{1}{d}\right)^{k(k-1)/2} \lambda_0, \forall k \geq 0;
        	\end{equation*}
	    \item for random BFGS/SR1 method, with probability at least $1-\delta$ over the random directions $\{\vu_k\}$, we could obtain
        	\begin{equation*}
        		\delta_k \leq \frac{4c d^3\eta}{\delta}\left(1-\frac{1}{d+1}\right)^{k} \text{  and  }
        		\lambda_{k} \leq \left(\frac{4c d^3\eta}{\delta}\right)^{k}\left(1-\frac{1}{d+1}\right)^{k(k-1)/2} \lambda_0, \forall k \geq 0.
        	\end{equation*}
	\end{itemize}
\end{theorem}

Finally, we combine with the linear convergence shown in Theorem 4.7 of \citet{rodomanov2021greedy} to give fair comparison of our superlinear convergence rates.
Under the SR1 update, unlike the measure $\sigma_\mA(\cdot)$ used by \cite{rodomanov2021greedy}, we employ a different measure $\tau_\mA(\cdot)$, requiring a stronger initial point condition to derive the convergence of $\{\lambda_k\}$ and $\{\delta_k\}$. Fortunately, we could obtain the same convergence bound with a slightly worse $k_0$ below.
The proof of Corollary \ref{cor:gel-sr1-bfgs} is given in Appendix \ref{app:miss-prove-random5}.

\begin{corollary}\label{cor:gel-sr1-bfgs}
	Suppose in Algorithm \ref{algo:general-quasi}, $\mG_0=L \mI_d$ and $\vx_0$ satisfies $M\lambda_0 \leq \frac{\ln\frac{3}{2}}{4\varkappa}$, that is, the initial condition here is weaker than that in Lemma \ref{lemma:gel-bfgs-sr1} when $\eta=\varkappa$. 
	Then we could obtain:
	\emph{1)} for the greedy BFGS/SR1 method,
	\[ \lambda_{k_0+k} \leq \left(1-\frac{1}{d}\right)^{k(k-1)/2} \cdot \left(\frac{1}{2}\right)^k \cdot \left(1-\frac{1}{2\varkappa}\right)^{k_0} \cdot \lambda_0, \text{ for all } k \geq 0, \]
	where $k_0=O\left((d+\varkappa)\ln(d\varkappa)\right)$; 
	\emph{2)} for the random BFGS/SR1 method, with probability at least $1-\delta$ over the random directions $\{\vu_k\}$, we have
	\[ \lambda_{k_0+k} \leq \left(1-\frac{1}{d+1}\right)^{k(k-1)/2} \cdot \left(\frac{1}{2}\right)^k \cdot \left(1-\frac{1}{2\varkappa}\right)^{k_0} \cdot \lambda_0, \text{ for all } k \geq 0, \]
	where $k_0=O\left((d+\varkappa)\ln(d\varkappa/\delta)\right)$ .
\end{corollary}
Therefore, both the greedy and random methods have nonasymptotic superlinear convergence rates. 
Additionally, our superlinear rates are condition-number-free compared to the rates in Corollary \ref{cor:random-gel} and the work of \cite{rodomanov2021greedy}.

\section{Discussion and Comparison}\label{sec:compa}

For better understanding the difference from the greedy quasi-Newton methods obtained in \cite{rodomanov2021greedy}, we give detailed comparison from the scope of the local convergence region and superlinear rates with $\mG_0=L\mI_d$, i.e., $\eta=\varkappa$ in our results. 

\textbf{Local Convergence Region.}
Because  we follow the proof of \cite{rodomanov2021greedy}'work, our linear convergence region is the same as theirs, i.e., $M\lambda_0 = O(\frac{1}{\varkappa})$.
Our superlinear convergence region of greedy/random BFGS (Lemma \ref{lemma:gel-bfgs-sr1}) and random Broyden (Lemma \ref{lemma:random-gel}) is the same as the one obtained in \citet[Theorem 4.9]{rodomanov2021greedy} for greedy Broyden method: $M\lambda_0 = O\left(\frac{1}{d\varkappa}\right)$. 
While our greedy/random SR1 (Lemma \ref{lemma:gel-bfgs-sr1}) needs a slight worse local region $M\lambda_0 = O\left(\frac{1}{d\varkappa^2}\right)$, because we use a different measure.

Different local regions show different warm-up iterations from linear rate region to superlinear rate region. Recall that the linear rates of these methods are the same as below:
\[ \lambda_k \stackrel{\eqref{eq:lambda-a}}{\leq} \left(1-\frac{1}{2\varkappa}\right)^k\lambda_0 \leq \exp\left\{-\frac{k}{2\varkappa}\right\}\lambda_0, \; \forall k \geq 0. \]
Thus, with beginning from $M\lambda_0 = O(\frac{1}{\varkappa})$, the linear rate lasts for $K_1=O(\varkappa \ln d)$ iterations for greedy/random Broyden and BFGS methods to make $M\lambda_{K_1} = O(\frac{1}{d\varkappa})$, but a slight worse $K_1=O(\varkappa \ln (d\varkappa))$ iterations for greedy/random SR1 methods to make $M\lambda_{K_1} = O(\frac{1}{d\varkappa^2})$.

\begin{table*}[t] 	
	\centering
	\begin{tabular}{cccc}
		\toprule
		\makecell{Quasi-Newton Methods \\ with $\mG_0=L\mI_d$} & \makecell{Local Region \\ ($M\lambda_0$)} & \makecell{Warm-up \\ ($K_1$)} & \makecell{Starting Moment \\ ($K_2$)}\\
		\midrule
		\makecell{Greedy/Random Broyden \\ \cite[]{rodomanov2021greedy} \\ (Lemma \ref{lemma:random-gel} and Theorem \ref{thm:random-p})} & $O\left(\dfrac{1}{d\varkappa}\right)$ & $O\left(\varkappa\ln d\right)$ & \makecell{$O\left(d\varkappa\ln(d\varkappa)\right)$ \\ $O\left(d\varkappa\ln(d\varkappa/\delta)\right)$} \\
		\midrule
		\makecell{Greedy/Random BFGS \\ (Lemma \ref{lemma:gel-bfgs-sr1} and Theorem \ref{thm:gre-ran-p})} & $O\left(\dfrac{1}{d\varkappa}\right)$ & $O\left(\varkappa\ln d\right)$ & \makecell{$O\left(d\ln(d\varkappa)\right)$ \\ $O\left(d\ln (d\varkappa/\delta) \right)$} \\
		\midrule
		\makecell{Greedy/Random SR1 \\ (Lemma \ref{lemma:gel-bfgs-sr1} and Theorem \ref{thm:gre-ran-p})} & $O\left(\dfrac{1}{d\varkappa^2}\right)$ & $O\left(\varkappa\ln (d\varkappa)\right)$ & \makecell{$O\left(d\ln(d\varkappa)\right)$ \\ $O\left(d\ln (d\varkappa/\delta)\right)$} \\
		\bottomrule
	\end{tabular}
	\caption{Comparison of 1) the local superlinear convergence region, 2) the warm-up iterations from linear rate region to superlinear rate region, and 3) the starting moment of superlinear rates at the local superlinear convergence region. 
	We all adopt $\mG_0=L\mI_d$ for brevity.
	For the randomized methods, the presented rates hold with probability at least $1-\delta$.}
	\label{table:res2}
\end{table*}

\textbf{Superlinear Rates.}
First, it is obvious that our greedy/random BFGS and SR1 methods have a faster rates than greedy/random Broyden methods because we improve the superlinear term from $(1-\frac{1}{d\varkappa})^{k(k-1)/2}$ to $(1-\frac{1}{d})^{k(k-1)/2}$.

Second, let us consider the starting moment of superlinear convergence.
For random Broyden methods, from Theorem \ref{thm:random-p}, we have the superlinear convergence is valid after 
\begin{equation}\label{eq:sp-ran-bro}
    K_2^{\mathrm{R-Broyden}}:=2(d\varkappa+1)\ln\frac{4d^3\varkappa^3}{\delta}+1
\end{equation}
iterations. Indeed, from Theorem \ref{thm:random-p}, for all $k \geq K_2^{\mathrm{R-Broyden}}$,
\[ \lambda_k \leq \left(\frac{4d^3\varkappa^3}{\delta}\right)^{k}  \left(1-\frac{1}{d\varkappa+1}\right)^{k(k-1)/2} \lambda_0 \leq \left[\frac{4d^3\varkappa^3}{\delta}  \exp\left\{-\frac{k-1}{2(d\varkappa+1)}\right\}\right]^k \lambda_0 (\stackrel{\eqref{eq:sp-ran-bro}}{\leq} \lambda_0). \]
Similarly, from Theorem \ref{thm:gre-ran-p}, we could obtain that the superlinear rates of our random BFGS and SR1 methods are valid after
\begin{equation*}
    K_2^{\mathrm{R-BFGS}}:=2(d+1)\ln\frac{4d^3\varkappa}{\delta}+1 \text{ and } K_2^{\mathrm{R-SR1}}:=2(d+1)\ln\frac{4d^3\varkappa^2}{\delta}+1
\end{equation*}
iterations, and the superlinear rates of our greedy BFGS and SR1 methods are valid after
\begin{equation*}
    K_2^{\mathrm{G-BFGS}}:=2d \ln (2d\varkappa)+1 \text{ and } K_2^{\mathrm{G-SR1}}:=2d \ln (2d\varkappa^2)+1
\end{equation*} 
iterations.
Moreover, based on \citet[Theorem 4.9]{rodomanov2021greedy}, we get
\begin{equation*}
    K_2^{\mathrm{G-Broyden}}:=2d\varkappa \ln (2d\varkappa)+1.
\end{equation*} 
Thus, our proposed greedy/random BFGS and SR1 methods improve the factor $O(d\varkappa\ln(d\varkappa))$ and $O(d\varkappa\ln(d\varkappa/\delta))$ of greedy/random Broyden methods to $O(d\ln(d\varkappa))$ and $O(d\ln(d\varkappa/\delta))$.

Third, we note that the local convergence regions of these methods are different from the discussion. 
Thus, we consider the whole convergent phase when $M\lambda_0 = O(1/\varkappa)$.
Based on Corollary \ref{cor:random-gel}, Corollary \ref{cor:gel-sr1-bfgs} and \citet[Theorem 4.9]{rodomanov2021greedy}, the starting moment of superlinear rates of our proposed greedy/random BFGS and SR1 methods at this time need $O((d+\varkappa)\ln(d\varkappa))$ (or $O((d+\varkappa)\ln(d\varkappa/\delta))$), which improves 
$O(d\varkappa\ln(d\varkappa))$ (or $O(d\varkappa\ln(d\varkappa/\delta))$) of greedy/random Broyden methods.
For brevity, we summarize the comparison discussed above to Tables \ref{table:res} and \ref{table:res2}.

\section{Numerical Experiments}\label{sec:exper}

In this section, we verify our theorems through numerical results for quasi-Newton methods. 
\citet[Section 5]{rodomanov2021greedy} have already compared their proposed greedy quasi-Newton methods with the classical quasi-Newton methods.
They showed that GrDFP, GrBFGS, GrSR1 (greedy DFP, BFGS, SR1 methods) with directions based on $\hat{\bm{u}}_\mA(\mG)$ (defined in Eq.~\eqref{eq:gre-u-ori}), have quite competitive convergence with the standard versions. 
They also presented the results for the randomized versions RaDFP, RaBFGS, RaSR1, which directly choose directions uniformly from the standard Euclidean sphere.
They found that the randomized methods are slightly slower than the greedy versions. 
However, the difference is not really significant.

The difference between our algorithms and their methods mainly comes from the greedy strategy for SR1 and the random strategy for BFGS\footnote{There is no difference in the random SR1 method compared to \citet{rodomanov2021greedy}, which directly selects random directions. And our greedy BFGS method is not efficient ($O(d^3)$ in each iteration) as we mentioned in Remark \ref{remark:bfgs-eff}. Thus we leave it out.}. Hence, we mainly focus on exhibiting our validity in these schemes. 
We refer to GrSR1v2 as our revised method and GrSR1v1 as the previous method (by adopting $\hat{\vu}_{\mA}(\mG)$). Similarly, we denote RaBFGSv2 that uses scaled directions ($\mL^\top_k\tilde\vu$) and RaBFGSv1 that directly uses random directions $\tilde\vu$ correspondingly. We choose the random directions from  $\mathrm{Unif}(\mathcal{S}^{d-1})$ in all randomized methods for brevity.

\begin{figure}[t]
	\centering
	\begin{subfigure}[b]{0.33\textwidth}
		\includegraphics[width=\linewidth]{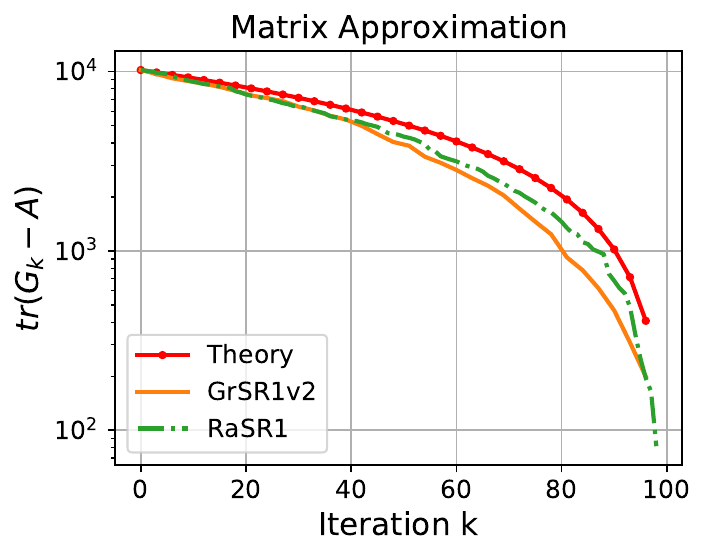}
		\caption{$d=100, \varkappa=2000$.} \label{fig:mat-quad-app1a}
	\end{subfigure}
	\begin{subfigure}[b]{0.33\textwidth}	
		\includegraphics[width=\linewidth]{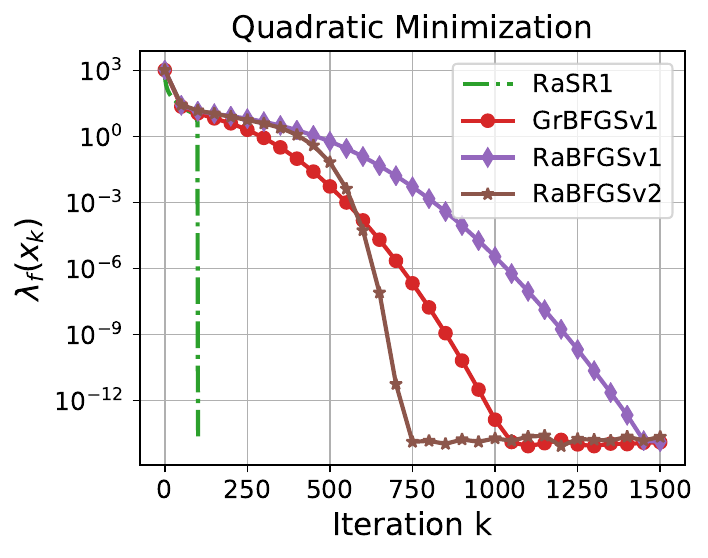}
		\caption{$d=100, \varkappa=2000$.} \label{fig:mat-quad-loss}	
	\end{subfigure}
	\\
	\begin{subfigure}[b]{0.32\textwidth}
		\includegraphics[width=\linewidth]{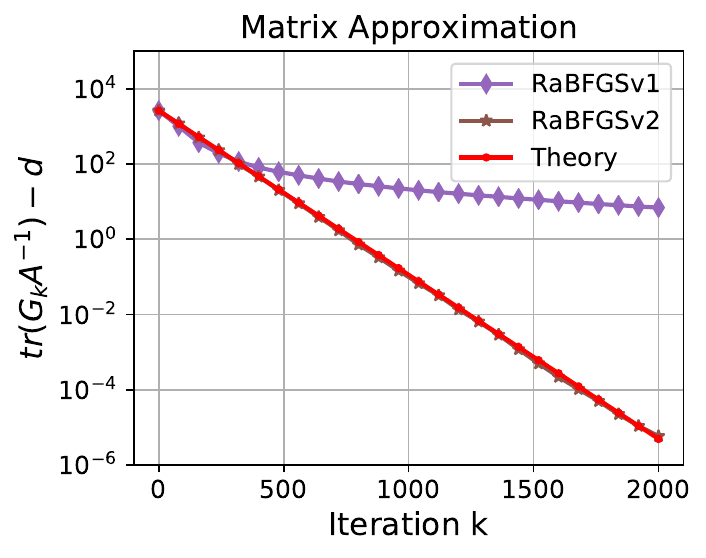}
		\caption{$d=100, \varkappa=200$.}
		\label{fig:mat-quad-app1c}	
	\end{subfigure}
	\begin{subfigure}[b]{0.32\textwidth}
		\includegraphics[width=\linewidth]{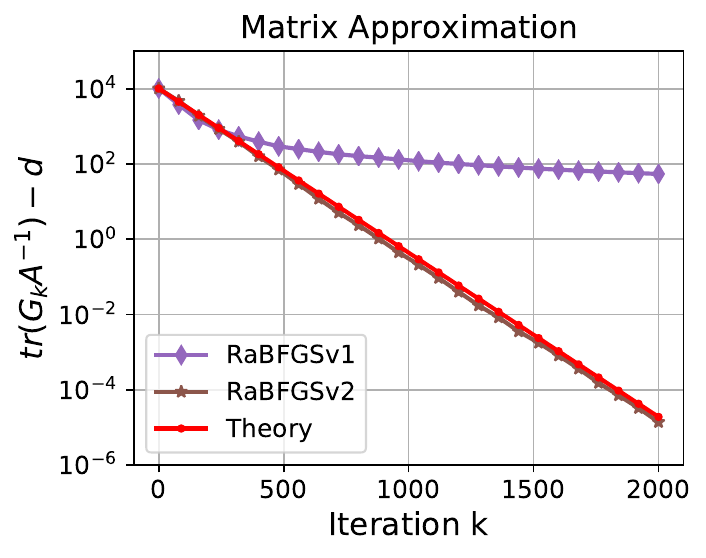}
		\caption{$d=100, \varkappa=2000$.}
		\label{fig:mat-quad-app1d}	
	\end{subfigure}
	\begin{subfigure}[b]{0.32\textwidth}
		\includegraphics[width=\linewidth]{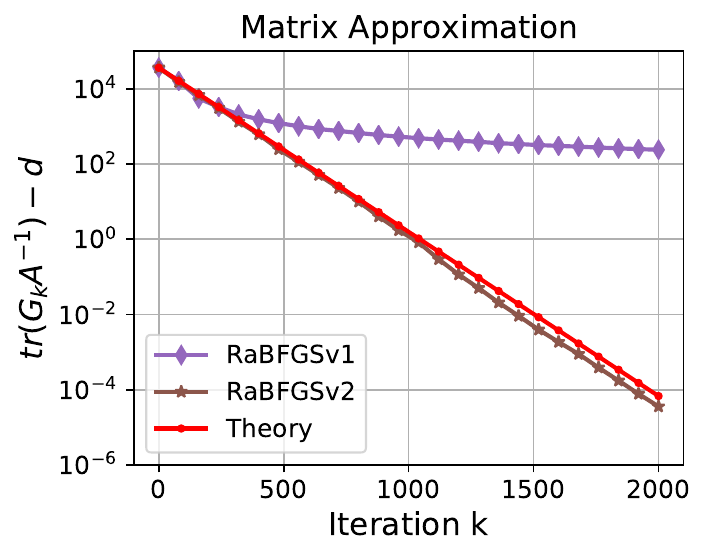}
		\caption{$d=100, \varkappa=20000$.}
		\label{fig:mat-quad-app1e}	
	\end{subfigure}
	\caption{(a, c, d, e) Comparison of different direction choosing methods under the SR1 or BFGS update for approximating a matrix $\mA$ that $\mu \mI_d \preceq \mA \preceq L\mI_d$ from $\mG_0=L\mI_d$. (a) The variation of $\tau_\mA(\mG_k)$ during Random SR1 (RaSR1) and our Greedy SR1 (GrSR1v2) update with nearly matched upper bound. (c,d,e) The variation of $\sigma_\mA(\mG_k)$ during our Random BFGS (RaBFGSv2) update and the original random version (RaBFGSv1) under various condition numbers. (b) Comparison of SR1 and BFGS methods for quadratic objective. Here we only depict RaSR1 method, while the other SR1-type methods share similar behavior.}
	\label{fig:mat-quad-app}
\end{figure}

\paragraph{Matrix approximation.} When using Algorithms \ref{algo:sr1-update} and \ref{algo:bfgs-update} for approximating a matrix $\mA \succ 0$, we show the measure as proved by Theorems \ref{thm:sr1-update} and \ref{thm:bfgs-update} in Figures \ref{fig:mat-quad-app1a}, \ref{fig:mat-quad-app1c},\ref{fig:mat-quad-app1d} and \ref{fig:mat-quad-app1e}. 
As Figure \ref{fig:mat-quad-app1a} depicts, our greedy and random SR1 updates (GrSR1v2 and RaSR1) share superlinear convergence rates under measure $\tau_\mA(\cdot)$, while our theoretical bound matches them well.
Moreover, Figures \ref{fig:mat-quad-app1c}, \ref{fig:mat-quad-app1d} and \ref{fig:mat-quad-app1e} describe the behavior of the random BFGS update under different condition numbers.
Our theory matches the linear convergence of measure $\sigma_\mA(\cdot)$ in our modified random BFGS update (RaBFGSv2) across different $\varkappa$s. 
While directly choosing a direction without scaling (RaBFGSv1) fails to give such bounds. Particularly, a large condition number could cause slow convergence of RaBFGSv1.
Hence, our methods provide effective ways of approaching a positive definite Hessian matrix. 

\paragraph{Quadratic minimization.} We also consider unconstrained quadratic minimization in Eq.~\eqref{eq:object-f} with the same positive definite matrix $\mA$ and a randomly selected vector $\vb\in\sR^d$. Running Algorithm \ref{algo:quad-quasi} with SR1 and BFGS updates, we obtain the superlinear convergence of $\lambda_f(\cdot)$ shown in Figure \ref{fig:mat-quad-loss}. Not surprisingly, our RaBFGSv2 runs faster than RaBFGSv1, while we also have the theoretical guarantee. At the same time, SR1-type methods converge to zero after $d+1$ steps because of $\mG_d=\mA$ almost surely. Here, we only depict the RaSR1 update, while the other SR1-type methods share similar behavior. Although our theoretical bound can not directly match the experiments due to the related initial terms $\tau_\mA(\mG_0)$ and $\sigma_\mA(\mG_0)$, the decay terms: $\left(1-k/d\right)$ vs. $\left(1-1/d\right)^k$ already show the superiority of the SR1 method over the BFGS method in the  quadratic minimization problem.  

\paragraph{Regularized Log-Sum-Exp.} Following the work of \citet{rodomanov2021greedy}, we present computational results for greedy and random quasi-Newton methods, applied to the following test function with $\mC = \left[\vc_1, \ldots, \vc_m\right]\in\sR^{d\times m}$, $b_1, \dots, b_m \in\sR$, and $\gamma>0$:
\begin{equation*}
f(\vx) := \ln \left(\sum_{j=1}^m e^{\vc_j^\top \vx-b_j}\right)+\frac{1}{2}\sum_{j=1}^m \left(\vc_j^\top\vx\right)^2+\frac{\gamma}{2}\left\|\vx\right\|^2, \vx\in\sR^d.
\end{equation*}
We need access to the gradient of function $f(\vx)$:
\begin{equation*}
\nabla f(\vx) = g(\vx)+\sum_{j=1}^m \left(\vc_j^\top\vx\right) \vc_j +\gamma\vx, \text{ with } g(\vx):=\sum_{j=1}^m \pi_j(\vx)\vc_j,
\end{equation*}
where
\begin{equation*}
\pi_j(\vx) := \frac{e^{\vc_j^\top \vx-b_j}}{\sum_{i=1}^me^{\vc_i^\top \vx-b_i}} \in [0, 1], j=1, \dots, m.
\end{equation*}
Moreover, given a point $\vx \in \sR^d$, we need to be able to perform the following two actions:
\begin{equation*}
\ve_i^\top[\nabla^2 f(\vx)]\ve_i = \sum_{j=1}^m \left(\pi_j(\vx)+1\right)(\vc_j^\top\ve_i)^2-(g(\vx)^\top\ve_i)^2 + \gamma, \forall 1 \leq i \leq d,
\end{equation*}
and for a given direction $\vh\in\sR^d$,
\begin{equation*}
\nabla^2 f(\vx) \cdot \vh = \sum_{j=1}^m \left(\pi_j(\vx)+1\right)\left(\vc_j^\top\vh\right)\vc_j-\left(g(\vx)^\top\vh\right)g(\vx) + \gamma \vh.
\end{equation*}
Thus both the above operations have a cost of $O(m d)$. Thus, the cost of one iteration
for all the methods is comparable.
Furthermore, note that
\[ \nabla^2 f(\vx) = \sum_{j=1}^m \left(\pi_j(\vx)+1\right)\vc_j\vc_j^\top-g(\vx)g(\vx)^\top + \gamma\mI_d. \]
We get the Lipschitz constant of $\nabla f(\vx)$ can be taken as
$L = 2\lambda_{\max}(\mC \mC^\top)+\gamma$, and $\varkappa = L/\gamma$.
As mentioned in the work of \citet[Section 5.1]{rodomanov2021greedy}, the strong self-concordancy parameter is $M=2$ with respect to the operator
$\sum_{j=1}^m \vc_j \vc_j^\top$.

We also adopt the same synthetic data as used by \citet[Section 5.1]{rodomanov2021greedy}.
First, we generate a collection of random vectors $\hat{\vc}_1, \ldots, \hat{\vc}_m$ with entries, uniformly distributed in the interval $ [-1, 1] $. Then we generate $ b_1, \dots, b_m$ from the same distribution. Using this data, we define
\[ \forall 1\leq j \leq m, \; \vc_j := \hat{\vc}_j -\nabla \hat{f}(\bm{0}), \text{ where } \hat{f}(\vx) := \ln \bigg(\sum_{j=1}^m e^{\hat{\vc}_j^\top \vx-b_j}\bigg). \]
Note that by construction,
\[ \nabla f(\bm{0}) = \frac{1}{\sum_{i=1}^me^{-b_i}}\sum_{j=1}^m e^{-b_j}\left(\hat{\vc}_j-\nabla \hat{f}(\bm{0})\right) =\bm{0}. \]
So the unique minimizer of our test function is $ \vx_* = \bm{0} $. The starting point $\vx_0$ for all methods is the same and generated randomly from the uniform distribution
on the standard Euclidean sphere of radius $1/d$ centered at the minimizer, i.e., $\vx_0\sim\text{Unif}\left(\frac{1}{d}\gS^{d-1}\right)$.
We compare $\norm{\nabla f(\vx_k)}$ obtained by different methods.

As Figure \ref{fig:logsumexp-a6a-reg} depicts, the BFGS-type methods are slower than the SR1-type methods, and the greedy algorithms converge more rapidly than the random algorithms. 
The only difference is that our RaBFGSv2 may have slower convergence behavior than RaBFGSv1 under a small $\varkappa$ in Figure \ref{fig:logsumexp-a6a-reg-1}. 
\begin{figure}[t]
	\centering
	\hspace{-7pt}
	\begin{subfigure}[b]{0.33\textwidth}
		\includegraphics[width=\linewidth]{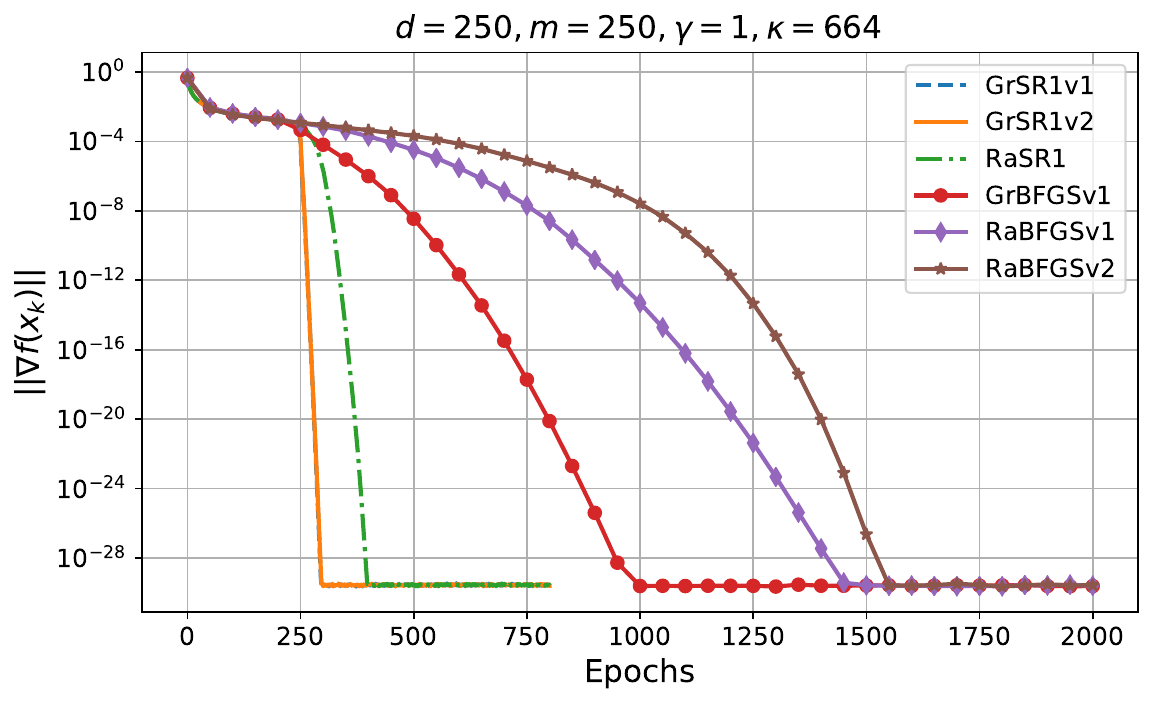}
		\caption{$\gamma=1, \varkappa=664$.}\label{fig:logsumexp-a6a-reg-1}
	\end{subfigure}
	\hspace{-7pt}
	\begin{subfigure}[b]{0.33\textwidth}
		\includegraphics[width=\linewidth]{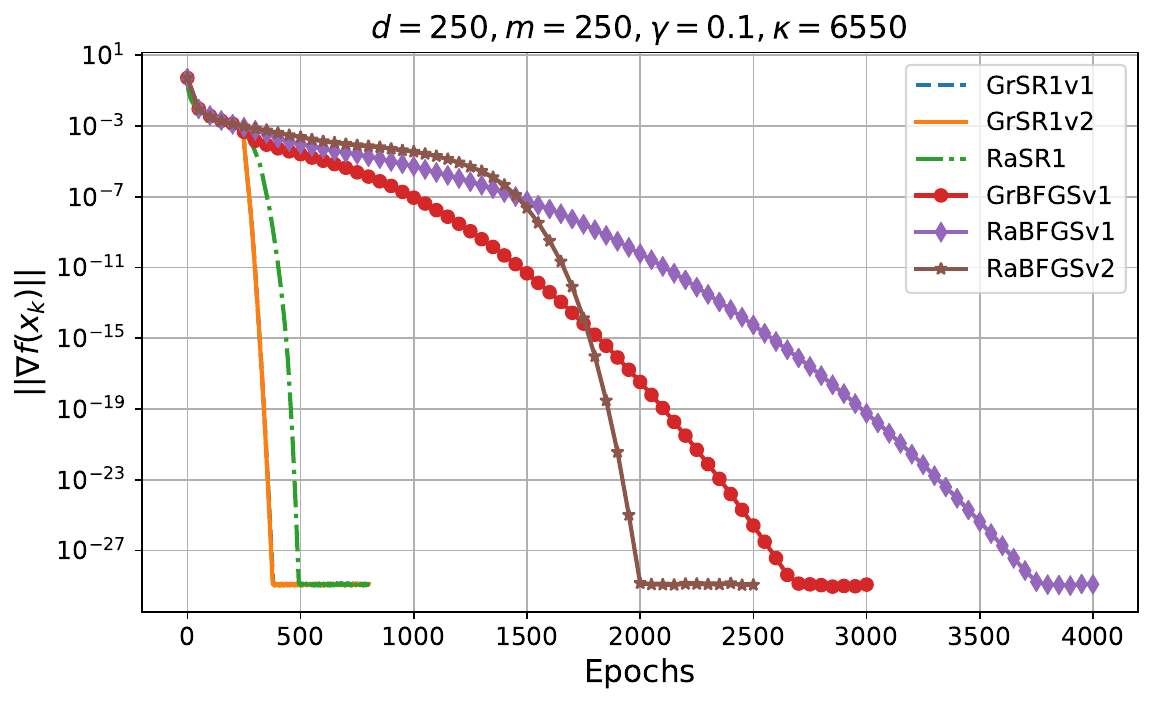}
		\caption{$\gamma=0.1, \varkappa=6550$.}\label{fig:logsumexp-a6a-reg-0.1}
	\end{subfigure}
	\hspace{-7pt}	
	\begin{subfigure}[b]{0.33\textwidth}
		\includegraphics[width=\linewidth]{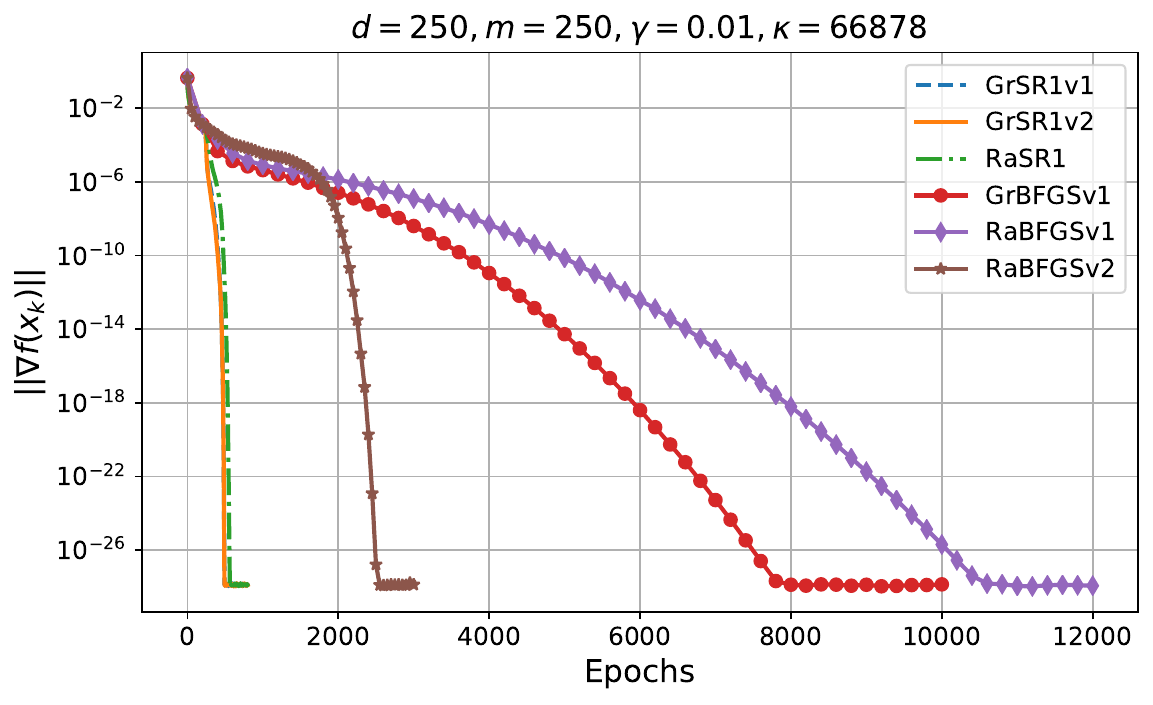}
		\caption{$\gamma=0.01, \varkappa=66878$.}\label{fig:logsumexp-a6a-reg-0.01}
	\end{subfigure}	
	\caption{Comparison of SR1 and BFGS updates for Regularized Log-Sum-Exp. The dimension $ d $, the number $ m $ of linear functions, the regularization coefficient $\gamma$ and condition number $\varkappa$ are displayed in the title of each graph. The lines of GrSR1v1 and GrSR1v2 are overlapped in each figure.}
	\label{fig:logsumexp-a6a-reg}
\end{figure}

We consider our scaled direction is more suitable for a constant Hessian matrix as the quadratic objective has. 
Thus we still have a better convergence rate in the last few iterations when Hessians are nearly unchanged in Figure \ref{fig:logsumexp-a6a-reg-1}. 
However, the Hessian varies drastically in the initial period. Thus there is less benefit under a more accurate Hessian approximation.
When applied to the ill-conditioning setting with a large $\varkappa$ in Figures \ref{fig:logsumexp-a6a-reg-0.1} and \ref{fig:logsumexp-a6a-reg-0.01}, we find our RaBFGSv2 could be faster than GrBFGSv1 and RaBFGSv1. This implies that  our proposed method has less dependence on the condition number $\varkappa$.

\paragraph{Regularized Logistic Regression.} 

Finally, we consider a common machine learning problem: $\ell_2$-regularized logistic regression, which has the objective as 
\begin{equation*}
f(\vw) = \sum_{i=1}^n \ln\left(1+e^{-y_i\vw^\top\vx_i}\right)+\frac{\gamma}{2}\|\vw\|^2, \; \vw\in\sR^d,
\end{equation*}
where $\mX = \left[\vx_1,\dots, \vx_n \right] \in \sR^{d \times n}$ are training samples, the corresponding labels are $y_1,\dots, y_n $ \\ $\in \{+1, -1\}$, and $\gamma>0$ is the regularization coefficient. 
The gradient of function $f(\vw)$ is
\[ \nabla f(\vw) = -\sum_{i=1}^n \frac{1}{1+e^{y_i\vw^\top\vx_i}} \cdot y_i \vx_i  + \gamma\vw, \; \vw \in\sR^d. \]
Moreover, given a point $\vw \in \sR^d$, we need to be able to perform the following two actions:
\[ \ve_j^\top [\nabla^2 f(\vw)]\ve_j = \sum_{i=1}^n \frac{ e^{y_i\vw^\top\vx_i}}{\left(1+e^{y_i\vw^\top\vx_i}\right)^2} \cdot (\vx_i^\top\ve_j)^2 + \gamma, \forall 1 \leq j \leq d, \]
and for a given $\vh \in\sR^d$,
\[ \nabla^2 f(\vw) \cdot \vh = \sum_{i=1}^n \frac{ e^{y_i\vw^\top\vx_i}}{\left(1+e^{y_i\vw^\top\vx_i}\right)^2}\cdot \left(\vx_i^\top \vh\right) \cdot \vx_i + \gamma\vh. \]
Thus, both the above operations have a cost of $O(nd)$.
Furthermore, note that
\[ \nabla^2 f(\vw) = \sum_{i=1}^n \frac{ e^{y_i\vw^\top\vx_i}}{\left(1+e^{y_i\vw^\top\vx_i}\right)^2} \cdot \vx_i\vx_i^\top + \gamma\mI_d, \; \vw\in\sR^d.
\]

\begin{figure}[t]
	\centering
	\begin{subfigure}[b]{0.4\textwidth}
		\includegraphics[width=\linewidth]{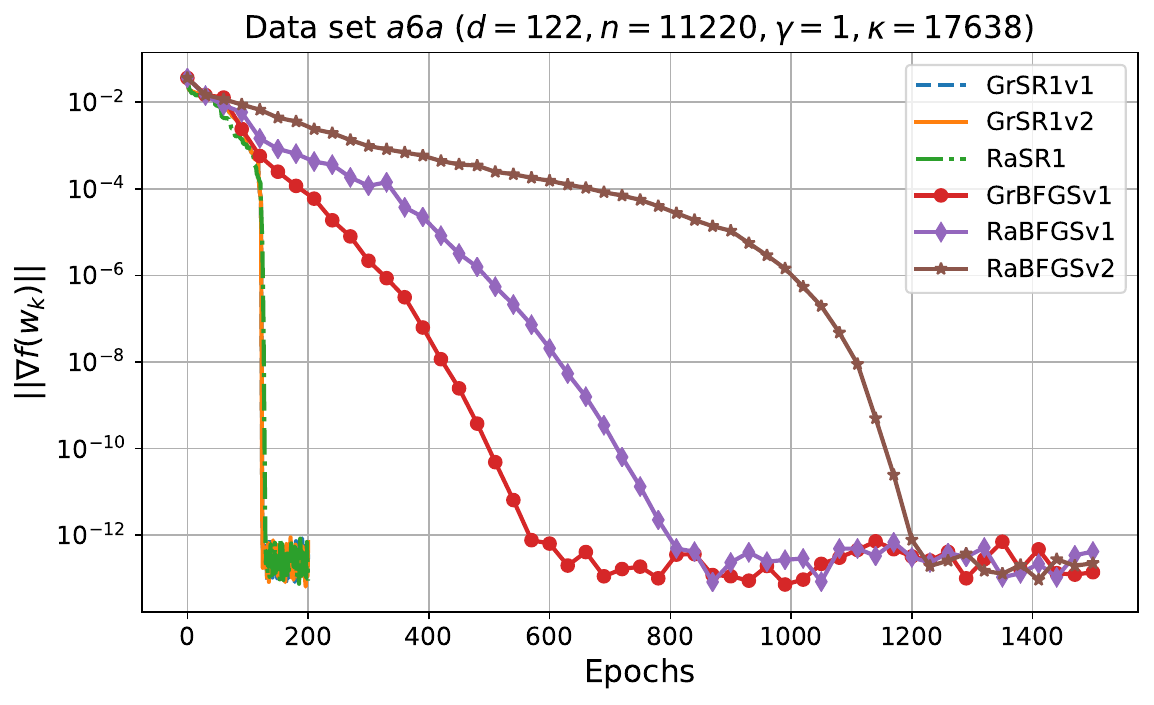}
		\caption{$\gamma=1, \varkappa=17638$.} \label{fig:logis-1}
	\end{subfigure}
	\begin{subfigure}[b]{0.4\textwidth}
		\includegraphics[width=\linewidth]{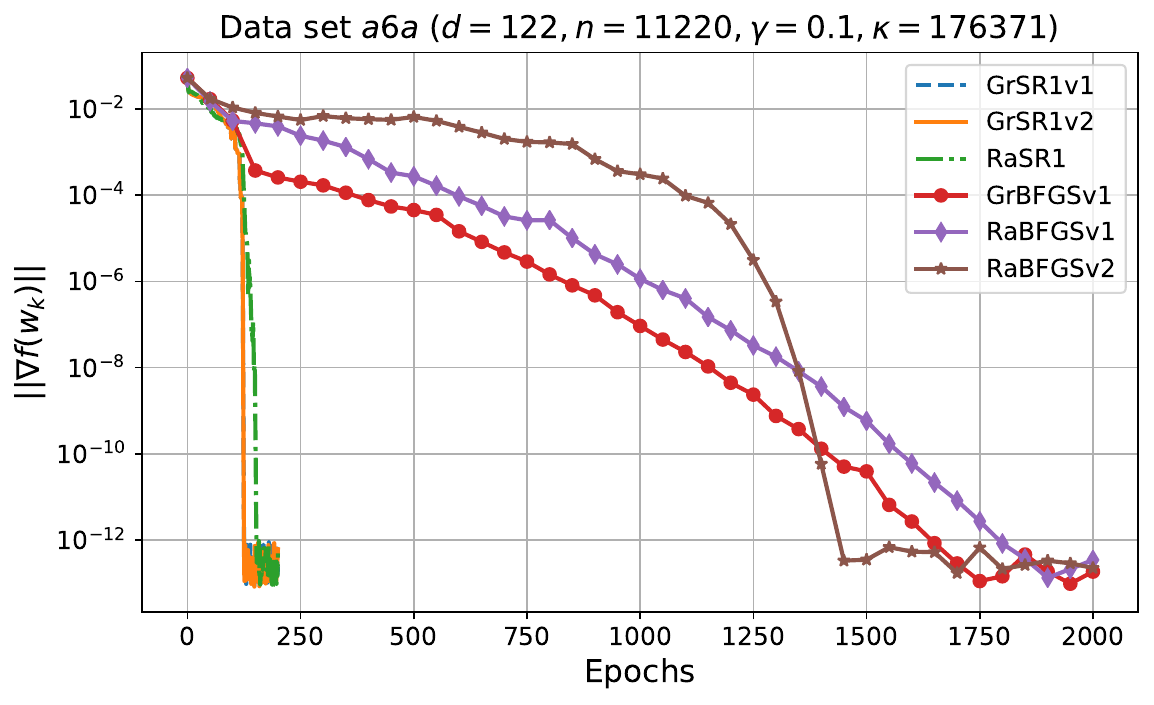}
		\caption{$\gamma=0.1, \varkappa=176371$.} \label{fig:logis-0.1}
	\end{subfigure}	
	\begin{subfigure}[b]{0.4\textwidth}
		\includegraphics[width=\linewidth]{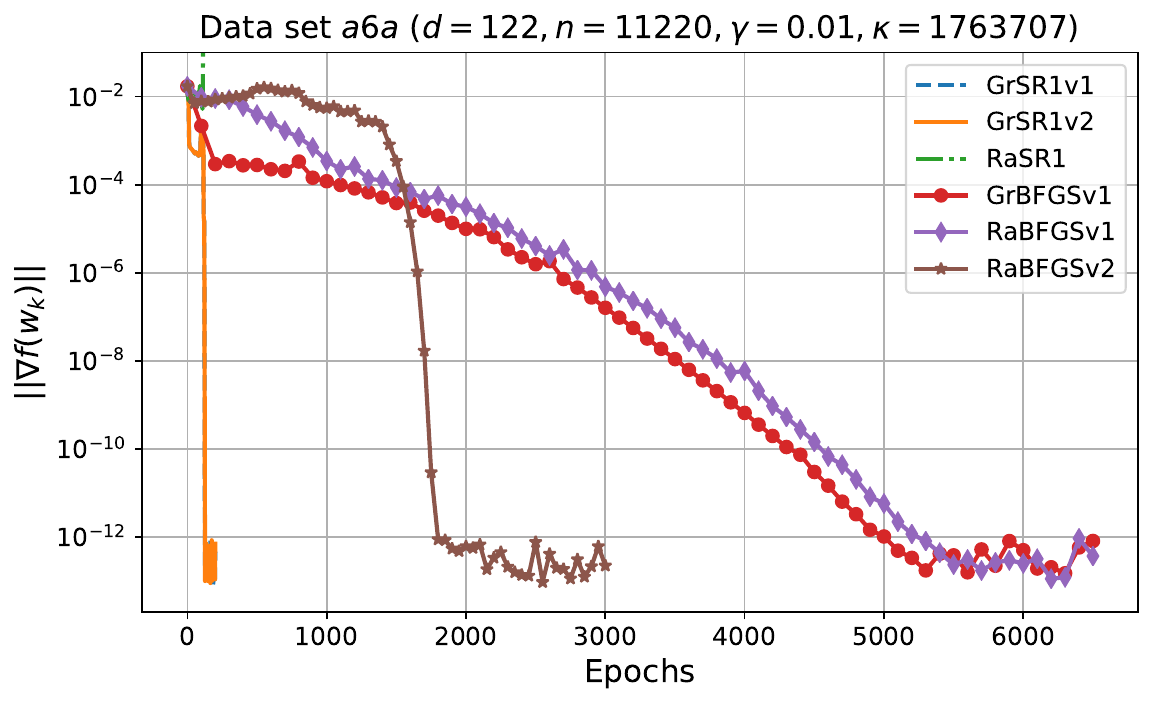}
		\caption{$\gamma=0.01, \varkappa=1763707$.} \label{fig:logis-0.01}
	\end{subfigure}	
	\begin{subfigure}[b]{0.4\textwidth}
		\includegraphics[width=\linewidth]{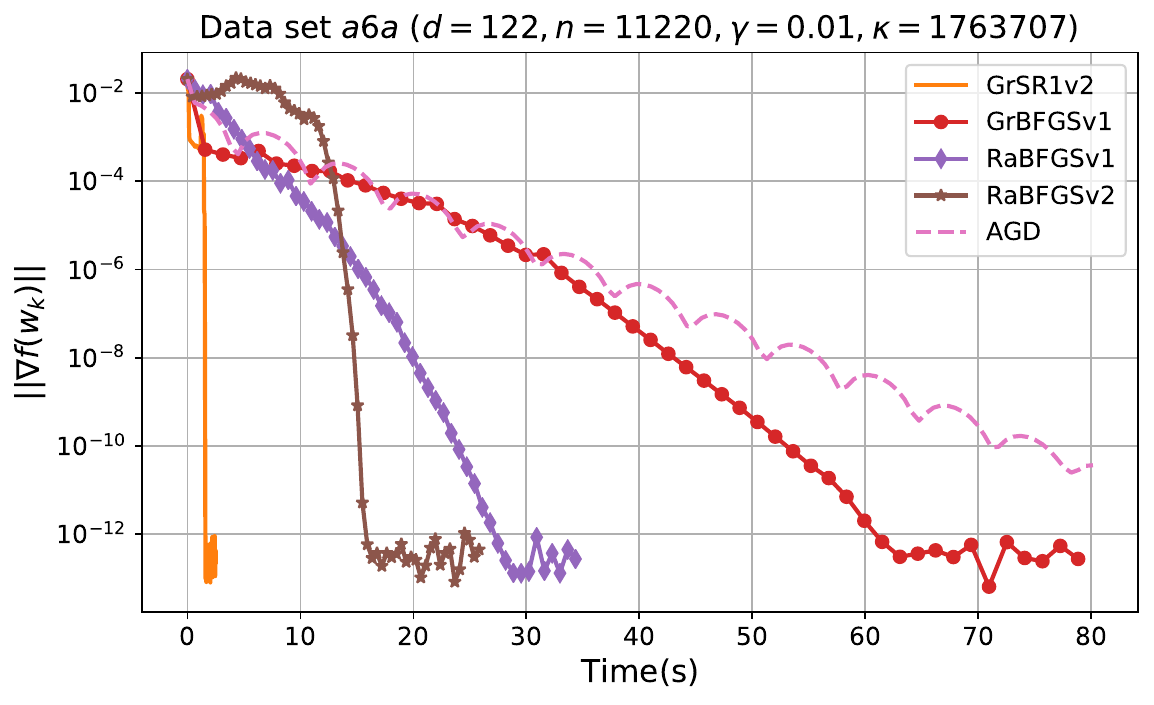}
	\caption{$\gamma=0.01, \varkappa=1763707$.} \label{fig:sr1-bfgs-time}
	\end{subfigure}	
	\caption{Comparison of SR1 and BFGS update for $\ell_2$-regularized logistic regression applied with `a6a' data from the LIBSVM collection of real-world datasets. 
	We list the name of dataset, the dimension $d$ and the condition number $\varkappa$ under the corresponding $\gamma$ in the title of each figure. (a,b,c) Comparison of convergence rates with various condition numbers. The lines of GrSR1v1 and GrSR1v2 are overlapped in some figures, and RaSR1 fails when $\gamma=0.01$ due to the unsuitable initialization and search directions.
	(d) Comparison of running time with AGD when $\gamma=0.01$.}
	\label{fig:sr1-bfgs}
\end{figure}

We obtain that the Lipschitz constant of $\nabla f(\vw)$ can be taken as $L = \frac{\lambda_{\max}(\mX\mX^\top)}{4} + \gamma$, and $\varkappa = L/\gamma$.
Additionally, we take data from the LIBSVM collection of real-world datasets for binary classification problems \citep{chang2011libsvm}. 
And we \textit{do not} apply the \textit{correction strategy} (i.e., $\tilde{\mG}_k=(1+Mr_k)\mG_k$) shown in Algorithm \ref{algo:general-quasi} recommended by \citet[Section 5.2]{rodomanov2021greedy}. 
In order to simulate the local convergence, we use the same initialization after running several standard Newton's steps to make the measure $\norm{\nabla f(\vw_0)}$ small (around $10^{-2} \sim 10^0$). 
We compare $\norm{\nabla f(\vw_k)}$ obtained by different methods. 

We show the results in Figure \ref{fig:sr1-bfgs}. 
As we can see, the general picture is the same as the Regularized Log-Sum-Exp. 
In particular, SR1-type methods are still faster than BFGS-type methods, and the greedy algorithms also converge more rapidly than the random algorithms.
GrBFGSv1 and RaBFGSv1 are faster than RaBFGSv2 in a small condition number case in Figure \ref{fig:logis-1}, but they become slower than RaBFGSv2 when the condition number becomes huge in Figures \ref{fig:logis-0.1} and \ref{fig:logis-0.01}.
Therefore, we think our RaBFGSv2 which uses scaled direction indeed has less dependence on the condition number as our theory shows. 

In addition, we also compare the running time of each method with a classical first-order method: accelerated gradient descent (AGD) following \cite{nesterov2003introductory}.
We run the standard AGD algorithm for $25000$ epochs with the same setting in Figure \ref{fig:logis-0.01}. 
As Figure \ref{fig:sr1-bfgs-time} shows, not surprisingly, we could discover the benefit of quasi-Newton methods in running time due to their superlinear convergence rates.
Furthermore, we find that GrBFGSv1 takes more time compared to the other methods because of the greedy step for searching directions. 
Thus, the greedy method is time-consuming as \cite{rodomanov2021greedy} discussed. 
But the random method solves this problem through a random choice of directions.
Moreover, we discover that random methods may fail for unsuitable initialization as the RaSR1 method in Figure \ref{fig:logis-0.01} shows, since we may encounter bad random directions during iterations, and our theoretical guarantee is also a probabilistic description. 
Hence, we recommend a mixture of greedy and random strategies in practice to balance the convergence and running complexity.

Overall, our proposed methods do not lose the superlinear convergence rates particularly in the large condition number schemes, while we also present the theoretical guarantee for these algorithms.

\section{Conclusion}\label{sec:conclude}
	In this work, we have addressed two open problems mentioned by \citet{rodomanov2021greedy}. 
	First, we have shown theoretical analysis of the random quasi-Newton methods, which also preserve a similar nonasymptotic superlinear convergence shown in the work of \citet{rodomanov2021greedy}. 
	Second, we have studied the behavior of two specific famous quasi-Newton methods: the SR1 and  BFGS methods.
	We have presented different greedy methods in contrast to the work of \citet{rodomanov2021greedy}, as well as the random version of these methods. 
	In particular, we have provided the faster Hessian approximation behavior and the condition-number-free (local) superlinear convergence rates applied to quadratic or strongly self-concordant objectives. Moreover, the experiments match our analysis  well.
	We hope that the theoretical analysis and the related work would be useful for understanding the explicit rates of quasi-Newton methods, and such convergence rates could benefit machine learning by developing new optimization methods.

\acks{We would like to thank the anonymous reviewers for their careful work and constructive comments that greatly help us improve the paper quality. We also thank an anonymous reviewer for pointing out efficient update in Proposition \ref{prop:l} and concise formulation in Lemmas \ref{lemma:random-gel} and \ref{lemma:gel-bfgs-sr1}, as well as providing a special lemma (Lemma \ref{lemma:general-update}) in Appendix.

Haishan Ye has been supported by the National Natural Science Foundation of China (No. 12101491).

}

\appendix 
\section{Auxiliary Lemmas and Theorems}
In the following, assume the objective $f(\vx)$ is an $M$-\textit{strongly self-concordant}, $\mu$-strongly convex and $L$-smooth function as Subsection \ref{subsec:random-gel} does.
\begin{lemma} \label{lemma:op-ud-l} \emph{\citep[Lemma 4.2]{rodomanov2021greedy}} 
	Let $\vx,\bm{y}\in\sR^d$, and  $r:=\|\bm{y}-\vx\|_{\vx}$. Then we have
	\begin{equation}\label{eq:fr}
	\frac{\nabla^2f(\vx)}{1+Mr}\preceq \nabla^2 f(\bm{y}) \preceq\left(1+Mr\right)\nabla^2f(\vx).
	\end{equation}
	Also, for $\mJ:= \int_{0}^{1}\nabla^2f(\vx+t(\bm{y}-\vx))dt$ and any $\bm{v}\in\{\vx, \bm{y}\}$, we have
	\begin{equation*}
	\frac{\nabla^2f(\bm{v})}{1+\frac{Mr}{2}} \preceq \mJ \preceq\left(1+\frac{Mr}{2}\right)\nabla^2f(\bm{v}).
	\end{equation*}
\end{lemma}

\begin{lemma} \label{lemma:op-ud} \emph{\citep[Lemmas 4.3 and 4,4]{rodomanov2021greedy}}
	Let $\vx\in\sR^d$, and a symmetric matrix $\mG$, such that $\nabla^2 f(\vx)\preceq \mG\preceq \eta \nabla^2 f(\vx)$, for some $\eta\geq 1$. 
	Let $\vx_+ \in \sR^d$, and $r= \|\vx_+-\vx\|_{\vx}$. Then we have
	\begin{equation}\label{eq:tg}
	    \tilde{\mG} := \left(1+Mr\right)\mG\succeq \nabla^2 f(\vx_+),
	\end{equation}
	and for all $\bm{u} \in \sR^d$ and $\tau\in[0,1]$, we have
	\begin{equation}\label{eq:tg-hes}
	    \nabla^2 f(\vx_+) \preceq \Broyd_{\tau}\left(\tilde{\mG}, \nabla^2 f(\vx_+), \bm{u}\right)\preceq \left[\left(1+Mr\right)^2\eta\right]\nabla^2 f(\vx_+).
	\end{equation}
	More specifically, if $\vx_+ = \vx-\mG^{-1}\nabla f(\vx)$, and letting $\lambda := \lambda_f(\vx)$ be such that $M\lambda\leq 2$, then,
	\begin{equation}\label{eq:lam-eta}
	    r \leq \lambda, \text{ and } \lambda_f(\vx_+) \leq \left(1+\frac{M\lambda}{2}\right) \frac{\eta-1+\frac{M\lambda}{2}}{\eta} \cdot \lambda.
	\end{equation}
\end{lemma}

\begin{theorem} \emph{\citep[Extension of][Theorem 4.7]{rodomanov2021greedy}}  \label{app:aux-thm}
	Suppose in Algorithm \ref{algo:general-quasi-ref}, a (random) initialization $\mG_0$ satisfies
	\begin{equation}\label{eq:init-G}
	    \nabla^2 f(\vx_0) \preceq \mG_0 \preceq \eta \nabla^2 f(\vx_0)
	\end{equation}
	for some $\eta \geq 1$, and the (random) initial point $\vx_0$ is sufficiently close to the solution:
	\begin{equation}\label{eq:init-linear}
	    M\lambda_0 \leq \frac{\ln\frac{3}{2}}{4\eta}.
	\end{equation}
	Then, for all $k\geq 0$, we have
	\begin{equation}\label{eq:hessian}
		\nabla^2 f(\vx_k) \preceq \mG_k \preceq e^{2M\sum_{i=0}^{k-1}\lambda_i}\eta \nabla^2 f(\vx_k) \preceq\frac{3\eta}{2}\nabla^2 f(\vx_k),
	\end{equation}
	and
	\begin{equation}\label{eq:lambda-a}
		\lambda_k \leq \left(1-\frac{1}{2\eta}\right)^k\lambda_0,
	\end{equation}
	where $\lambda_k$ is defined in Eqs.~\eqref{eq:lambda} and \eqref{eq:lam-sig}.
\end{theorem}

\begin{algorithm}[t]
	\caption{Quasi-Newton Method \cite[Scheme (4.17)]{rodomanov2021greedy}}
	\begin{algorithmic}[1]
		\STATE Initialization: Choose $\mG_0\succeq \nabla^2 f(\vx_0)$ (such as $\mG_0=L\mI_d$).
		\FOR{$ k \geq 0 $}
		\STATE Update $\vx_{k+1} = \vx_{k} - \mG_k^{-1}\nabla f(\vx_k)$.
		\STATE Compute $r_k = \|\vx_{k+1}-\vx_k\|_{\vx_k}$ and set $\tilde{\mG}_k = \left(1+Mr_k\right)\mG_k$.
		\STATE Choose $\bm{u}_k\in\sR^d$ and $\tau_k\in[0, 1]$.
		\STATE Compute $\mG_{k+1}=\Broyd_{\tau_k}\left(\tilde{\mG}_k,\nabla^2 f(\vx_{k+1}), \bm{u}_k\right)$.
		\ENDFOR
	\end{algorithmic}
	\label{algo:general-quasi-ref}
\end{algorithm}
\begin{proof}
    The proof is similar as Theorem 4.7 by \citet{rodomanov2021greedy}. We give the detail for completeness.
    
    From Eq.~\eqref{eq:init-G}, both Eqs.~\eqref{eq:hessian} and \eqref{eq:lambda-a} are satisfied for $k=0$.
    Now let $k\geq 0$, and suppose Eqs.~\eqref{eq:hessian} and \eqref{eq:lambda-a} have already been proved for all $0\leq k'\leq k$. 
    Denote 
    \begin{equation}\label{eq:etak}
        \eta_k := e^{2M\sum_{i=0}^{k-1}\lambda_i}\eta.
    \end{equation}
    Then by inductive hypothesis, we have
    \begin{equation}\label{eq:hess}
        \nabla^2 f(\vx_k) \preceq \mG_k \preceq \eta_k \nabla^2 f(\vx_k).
    \end{equation}
    Note that
    \begin{equation}\label{eq:m-lam-ran-1}
	    M\sum_{i=0}^k \lambda_i \stackrel{\eqref{eq:lambda-a}}{\leq} M\lambda_0\sum_{i=0}^{k}\left(1-\frac{1}{2\eta}\right)^i \leq 2\eta M \lambda_0 \stackrel{\eqref{eq:init-linear}}{\leq} \frac{\ln\frac{3}{2}}{2}<1.
	\end{equation}
	Hence, $M\lambda_k<2$, and by Lemma \ref{lemma:op-ud}, we have
	\begin{equation}\label{eq:r-lam-ran-1}
		r_k:=\|\vx_{k+1}-\vx_k\|_{\vx_k} \stackrel{\eqref{eq:lam-eta}}{\leq} \lambda_k,
	\end{equation}
	and
	\begin{equation}\label{eq:lam-re-ran-1}
	    \lambda_{k+1} \stackrel{\eqref{eq:lam-eta}\eqref{eq:hess}}{\leq} \left(1+\frac{M\lambda_k}{2}\right) \frac{\eta_k-1+\frac{M\lambda_k}{2}}{\eta_k}\lambda_k \leq \left(1+\frac{M\lambda_k}{2}\right) \left(1-\frac{1-\frac{M\lambda_k}{2}}{\eta_k}\right)\lambda_k.
	\end{equation}
	Using the inequality $1-t \geq e^{-2t}$ when $ 0 \leq t \leq 1/2$, we get
	\[ \frac{1-\frac{M\lambda_k}{2}}{\eta_k} \stackrel{\eqref{eq:m-lam-ran-1}}{\geq} e^{-M\lambda_k} \eta_k^{-1} \stackrel{\eqref{eq:etak}}{\geq } e^{-2M\sum_{i=0}^{k}\lambda_i}\eta^{-1} \stackrel{\eqref{eq:m-lam-ran-1}}{\geq} \frac{2}{3\eta}. \]
	Moreover, noting that $M\lambda_k \stackrel{\eqref{eq:lambda-a}}{\leq} M\lambda_0 \stackrel{\eqref{eq:init-linear}}{\leq} \frac{1}{8\eta}$, we obtain
	\[ \left(1+\frac{M\lambda_k}{2}\right) \left(1-\frac{1-\frac{M\lambda_k}{2}}{\eta_k}\right) \leq \left(1+\frac{1}{16\eta}\right)\left(1-\frac{2}{3\eta}\right) \leq 1-\frac{2}{3\eta}+\frac{1}{16\eta} \leq 1-\frac{1}{2\eta}. \]
	Consequently,
	\[ \lambda_{k+1} \stackrel{\eqref{eq:lam-re-ran-1}}{\leq} \left(1-\frac{1}{2\eta}\right) \lambda_k \stackrel{\eqref{eq:lambda-a}}{\leq} \left(1-\frac{1}{2\eta}\right)^{k+1} \lambda_0. \]
	Finally, from Eq.~\eqref{eq:tg-hes} in Lemma \ref{lemma:op-ud}, it follows that
	\begin{eqnarray*}
	    \nabla^2 f(\vx_{k+1}) &\preceq& \mG_{k+1} \preceq \left(1+Mr_k\right)^2 \eta_k \nabla^2 f(\vx_{k+1}) \stackrel{\eqref{eq:r-lam-ran-1}}{\preceq} \left(1+M\lambda_k\right)^2 \eta_k \nabla^2 f(\vx_{k+1}) \\
	    & \preceq & e^{2M\lambda_k} \eta_k \nabla^2 f(\vx_{k+1}) 
	    \stackrel{\eqref{eq:etak}}{=} e^{2M\sum_{i=0}^{k}\lambda_i} \eta \nabla^2 f(\vx_{k+1}) \stackrel{\eqref{eq:m-lam-ran-1}}{\preceq} \frac{3\eta}{2} \nabla^2 f(\vx_{k+1}).
	\end{eqnarray*}
	Thus, Eqs.~\eqref{eq:hessian} and \eqref{eq:lambda-a} are valid for index $k+1$, and we can continue by induction.
\end{proof}
\begin{remark}
	Note that the choice of $\{\vu_k\}$ and the update rule in Algorithm \ref{algo:general-quasi-ref} are arbitrary, thus Algorithms \ref{algo:general-random} and \ref{algo:general-quasi} can be viewed as special cases of Algorithm \ref{algo:general-quasi-ref}. 
	Therefore, Theorem \ref{app:aux-thm} always holds as long as the initial point $\vx_0$ is always sufficiently close to the solution based on the initial approximate matrix $\mG_0$, and Eq.~\eqref{eq:lambda-a} holds without any expectations.
\end{remark}

\begin{lemma} \emph{\citep[Extension of][Lemma 4.8]{rodomanov2021greedy}} \label{lemma:general-update}
\quad \\
    Following the update in Algorithms \ref{algo:general-random} or \ref{algo:general-quasi}, we can obtain 
    \begin{equation}\label{eq:delta-k}
        \forall k\geq 0, \nabla^2 f(\vx_k) \preceq \mG_k \preceq (1+\delta_k)\nabla^2 f(\vx_k),
    \end{equation}
    where $\delta_k$ is a random sequence which satisfies the following recurrence:
    \begin{equation}\label{eq:delta-up}
        \E_{\vu_k}\delta_{k+1} \leq \left(1-t^{-1}\right)\left(1 + Mr_k\right)^2\left(\delta_k+2c d Mr_k\right)
    \end{equation}
    for some constants $c, t \geq 1$. Particularly, 
    \begin{enumerate}
        \item[\emph{1)}] for the Broyden method in Algorithm \ref{algo:general-random}, one has $\delta_k=\sigma_k, t=d\varkappa, c=1$;
        \item[\emph{2)}] for the BFGS method in Algorithm \ref{algo:general-quasi}, one has $\delta_k=\sigma_k, t=d, c=1$;
        \item[\emph{3)}] for the SR1 method in Algorithm \ref{algo:general-quasi}, one has $\delta_k=d\varkappa\tau_k/\tr[\nabla^2 f(\vx_k)], t=d, c=\varkappa$.
    \end{enumerate}
    Here, $r_k := \|\vx_{k+1}-\vx_k\|_{\vx_k}, \sigma_k$ and $\tau_k$ are defined in Eq.~\eqref{eq:lam-sig}.
\end{lemma}

\begin{proof}
    From the update rule in Algorithms \ref{algo:general-random} and \ref{algo:general-quasi}, we apply Lemma \ref{lemma:op-ud} to obtain 
    \begin{equation}\label{eq:g-hes}
        \mG_{k} \stackrel{\eqref{eq:tg}}{\succeq} \nabla^2 f(\vx_k), \; \forall k \geq 0.
    \end{equation} 
    Now for all $k \geq 0$, we define $\eta_k = \norm{[\nabla^2 f(\vx_k)]^{-1/2}\mG_k[\nabla^2 f(\vx_k)]^{-1/2}}$. Then $\eta_k\stackrel{\eqref{eq:g-hes}}{\geq} 1$ and $\mG_k \preceq \eta_k \nabla^2 f(\vx_k)$. Hence,
	\begin{equation*}
    	\begin{aligned}
    	    \eta_k-1 &= \norm{[\nabla^2 f(\vx_k)]^{-1/2}(\mG_k-\nabla^2 f(\vx_k))[\nabla^2 f(\vx_k)]^{-1/2}} \\
    	    &\leq \tr\left([\nabla^2 f(\vx_k)]^{-1/2}\left(\mG_k-\nabla^2 f(\vx_k)\right)[\nabla^2 f(\vx_k)]^{-1/2} \right) \stackrel{\eqref{eq:lam-sig}}{=} \sigma_{k}.
    	\end{aligned}
	\end{equation*}
	Moreover, from $\mu\mI_d \preceq \nabla^2 f(\vx_k) \preceq L\mI_d, \forall k\geq 0$, we also have 
	\begin{equation*}
    	\begin{aligned}
    	    \eta_k-1 &= \norm{[\nabla^2 f(\vx_k)]^{-1/2}(\mG_k-\nabla^2 f(\vx_k))[\nabla^2 f(\vx_k)]^{-1/2}} \leq \frac{\|\mG_k-\nabla^2 f(\vx_k)\|}{\mu} \\
    	    &\leq \frac{\tr\left[\mG_k-\nabla^2 f(\vx_k)\right]}{\mu} \stackrel{\eqref{eq:lam-sig}}{=} \frac{\tau_k}{\mu} \leq \frac{d L}{\tr[\nabla^2 f(\vx_k)]} \cdot \frac{\tau_k}{\mu} =\frac{d\varkappa\tau_k}{\tr[\nabla^2 f(\vx_k)]}.
    	\end{aligned}
	\end{equation*}
	Therefore, the choices of $\delta_k=\sigma_k$ and $\frac{d\varkappa\tau_k}{\tr[\nabla^2 f(\vx_k)]}$ are valid to guarantee Eq.~\eqref{eq:delta-k}.
	Next, we deduce Eq.~\eqref{eq:delta-up} based on the specific choice of $\delta_k$.
	
	1) For the Broyden method in Algorithm \ref{algo:general-random}, by Theorem \ref{thm:rand-update}, one step update gives 
	\[ \E_{\vu_k} \sigma_{k+1} \stackrel{\eqref{eq:lam-sig}}{=} \E_{\vu_k}  \sigma_{\nabla^2f(\vx_{k+1})}\left(\mG_{k+1}\right) \stackrel{\eqref{eq:random-k}}{\leq} \left(1-\frac{1}{d\varkappa}\right) \sigma_{\nabla^2f(\vx_{k+1})}(\tilde{\mG}_k), \forall k \geq 0. \]
	We deduce the result by bounding the last term as below: 
	\begin{eqnarray}
	        \sigma_{\nabla^2f(\vx_{k+1})}(\tilde\mG_{k})  &\stackrel{\eqref{eq:sigmaA}}{=}& \tr\left( [\nabla^2f(\vx_{k+1})]^{-1}\tilde{\mG}_k \right) {-} d = \left(1 {+} Mr_k\right)\tr\left( [\nabla^2f(\vx_{k+1})]^{-1}\mG_k \right) {-} d \nonumber \\
	        &\stackrel{\eqref{eq:fr}}{\leq}& \left(1 {+} Mr_k\right)^2 \tr\left( [\nabla^2f(\vx_{k})]^{-1}\mG_k \right) {-}d \stackrel{\eqref{eq:lam-sig}}{=} \left(1 {+} Mr_k\right)^2\left(\sigma_k {+} d\right) {-} d \nonumber \\
	        &=& \left(1 {+} Mr_k\right)^2\sigma_k {+} d\left[2Mr_k {+} (Mr_k)^2\right] \leq
	        \left(1 {+} Mr_k\right)^2\left(\sigma_k {+} 2d Mr_k\right). \label{eq:sig-bound}
	    \end{eqnarray}
	
	2) For the BFGS method in Algorithm \ref{algo:general-quasi}, by Theorem \ref{thm:bfgs-update}, one step update gives 
	\[ \E_{\vu_k} \sigma_{k+1} \stackrel{\eqref{eq:sigma-k}}{\leq} \left(1-\frac{1}{d}\right) \sigma_{\nabla^2f(\vx_{k+1})}(\tilde{\mG}_k), \forall k \geq 0. \]
	The remaining proof is the same as Eq.~\eqref{eq:sig-bound}.
    
    3) For the SR1 method in Algorithm \ref{algo:general-quasi}, from Theorem \ref{thm:sr1-update}, one step update gives 
    \[ \E_{\vu_k}\tau_{k+1} \stackrel{\eqref{eq:tau-k}}{\leq} \left(1-\frac{1}{d}\right)\tau_{\nabla^2f(\vx_{k+1})}(\tilde{\mG}_k) , \forall k \geq 0. \]
	Now we can bound the last term as below: 
	\begin{eqnarray*}
    	\tau_{\nabla^2f(\vx_{k+1})}(\tilde{\mG}_k) &\stackrel{\eqref{eq:lam-sig}}{=}& \tr\left(\tilde{\mG}_{k}-\nabla^2 f(\vx_{k+1})\right) \stackrel{\eqref{eq:fr}}{\leq}  \tr\left(\left(1+Mr_k\right)\mG_k-\frac{\nabla^2 f(\vx_{k})}{1+Mr_k}\right) \\
    	&\stackrel{\eqref{eq:lam-sig}}{=}& \left(1+Mr_k\right)\tau_k+\left(1+Mr_k-\frac{1}{1+Mr_k}
    	\right)\tr\left[\nabla^2 f(\vx_k)\right] \\
    	&\stackrel{(i)}{\leq} &
    	\left(1+Mr_k\right)\tau_{k}+2Mr_k\tr\left[\nabla^2 f(\vx_{k})\right] \\
    	&\leq &\left(1+Mr_k\right)\left(\frac{\tau_{k}}{\tr[\nabla^2 f(\vx_k)]}+2Mr_k\right)\tr\left[\nabla^2 f(\vx_{k})\right] \\ &\stackrel{\eqref{eq:fr}}{\leq}& \left(1+Mr_k\right)^2\left(\frac{\tau_{k}}{\tr[\nabla^2 f(\vx_k)]}+2Mr_k\right)\tr\left[\nabla^2 f(\vx_{k+1})\right],
	\end{eqnarray*}
	where $(i)$ uses inequality $1+a-\frac{1}{1+a} \leq 2a, a \geq 0$.
	Thus, by replacing $\delta_k=\frac{d\varkappa\tau_k}{\tr[\nabla^2 f(\vx_k)]}$, we obtain
	\[ \E_{\vu_{k}}\delta_{k+1} \leq \left(1-\frac{1}{d}\right)\left(1+Mr_k\right)^2\left(\delta_k+2\varkappa d Mr_k\right). \]
\end{proof}

\begin{lemma}\label{lemma:p}
    Suppose there exist some constants $a \geq 0, t > 1$, and a nonnegative random sequence $\{\mX_k\}$ satisfies
    \[ \E \mX_k \leq a\left( 1-\frac{1}{t} \right)^k, \forall k \geq 0. \]
    Then for any $\delta \in(0, 1)$, with probability at least $1-\delta$, we have 
    \[ \forall k \geq 0, \mX_k \leq \frac{a t^2}{\delta} \left(1-\frac{1}{1+t}\right)^{k}. \]
\end{lemma}
\begin{proof}
	If $a=0$, then by $\mX_k \geq 0, \forall k\geq 0$, we can see $\mX_k=0$, a.s. Then the results trivial hold.
	Now we consider $a>0$. Noting that $\mX_k \geq 0$ and using Markov's inequality, we have for any $\epsilon>0$, 
	\begin{equation}\label{eq:pp}
	\sP\left(\mX_k > \frac{a}{\epsilon} \left( 1-\frac{1}{t} \right)^k \right) \leq \sP\left(\mX_k \geq \frac{a}{\epsilon} \left( 1-\frac{1}{t} \right)^k \right) \leq \frac{\E \mX_k}{a\left( 1-\frac{1}{t} \right)^k} \cdot \epsilon \leq \epsilon.
	\end{equation}
	Choosing $\epsilon_k = \delta (1-q)q^k$ for some $q\in(0,1)$ and applying the union bound, we obtain
	\begin{equation*}
	    \sP\left(\exists k \geq 0, \mX_k > \frac{a}{\epsilon_k} \left(1-\frac{1}{t}\right)^{k}\right) 
	    \leq \sum_{k=0}^{+\infty} \sP\left(\mX_k > \frac{a}{\epsilon_k} \left(1-\frac{1}{t}\right)^{k} \right) \stackrel{(\ref{eq:pp})}{\leq} \sum_{k=0}^{+\infty} \epsilon_k = \delta.
	\end{equation*}
	Therefore, with probability at least $1-\delta$, we have 
	\[ \forall k \geq 0, \mX_k \leq \frac{a}{\delta (1-q)q^k} \left(1-\frac{1}{t}\right)^{k}. \]
	If we let $q = 1-\frac{1}{t^2}$, then we can simplify the above inequality into
	\[ \forall k \geq 0, \mX_k \leq \frac{a t^2}{\delta} \left(1+\frac{1}{t}\right)^{-k} = \frac{a t^2}{\delta} \left(1-\frac{1}{t+1}\right)^{k}. \]
\end{proof}

\section{Missing Proofs of Matrix Approximation}

\subsection{Proof of Theorem \ref{thm:rand-update}}\label{app:miss-prove-random1}
\begin{proof}
	From Lemma \ref{lemma:monotonic}, we have $\mG_k\succeq\mA$ and $\sigma_k\stackrel{\eqref{eq:lam-sig}}{=}\sigma_{\mA}(\mG_{k})\geq 0, \forall k\geq 0$. 
	Moreover, from Eqs.~\eqref{eq:sigma-bound} and \eqref{eq:randomu}, we get
	\begin{eqnarray*}
		\E_{\vu_k} \sigma_{k+1} &\stackrel{\eqref{eq:sigma-bound}}{\leq}& \sigma_k - \frac{1}{L} \tr\left[(\bm{G}_k-\mA) \cdot \E_{\vu_k} \frac{\bm{u}_k\bm{u}_k^\top }{\bm{u}_k^\top\bm{u}_k} \right] \stackrel{\eqref{eq:randomu}}{=} \sigma_k - \frac{1}{d L} \tr(\bm{G}_k-\mA) \\
		&\stackrel{\eqref{eq:ass-A}}{\leq}& \sigma_k - \frac{\mu}{d L} \tr\left[(\bm{G}_k-\mA)\mA^{-1}\right] = \left(1-\frac{1}{d\varkappa}\right)\sigma_k.
	\end{eqnarray*}
	Finally, taking the expectation of all randomness, we get
	\begin{equation*}
	    \E \sigma_k \leq \left(1-\frac{1}{d\varkappa}\right) \E \sigma_{k-1} \leq \cdots \leq \left(1-\frac{1}{d\varkappa}\right)^{k} \E \sigma_0, \forall k\geq 1.
	\end{equation*}
\end{proof}

\subsection{Proof of Theorem \ref{thm:sr1-update}}\label{app:miss-prove1}
\begin{proof}
    Denoting $\mG_{k+1} := \text{SR1}(\mG_k, \mA, \vu_k)$ and $\mR_k := \mG_k-\mA , \forall k \geq 0$,  we have the update rule:
    \begin{equation}\label{eq:Rk}
	\mR_{k+1} \stackrel{\eqref{eq:sr1}}{=} 
	\begin{cases}
	\mR_k - \frac{\mR_k\vu_k \vu_k^\top \mR_k}{\vu_k^\top\mR_k\vu_k}, & \text{ if } \mR_k\vu_k \neq \bm{0}; \\
	\mR_k, & \text{ otherwise}.
	\end{cases}
	\end{equation}
	We also have $\mR_k \succeq \bm{0}$ by Lemma \ref{lemma:monotonic} since $\mR_0 \succeq \bm{0}$. It is easily seen that
	1) $\mathrm{Ker}(\mR_k) \subseteq \mathrm{Ker}(\mR_{k+1}) := \{\vv: \mR_{k+1}\vv = \bm{0}\}$, and 2) $\vu_k \in \mathrm{Ker}(\mR_{k+1})$ from Eq.~\eqref{eq:Rk}.
    
    \textbf{For the greedy method}, we denote $\bar\vu_k = \bar\vu_A(\mG_k), \forall k \geq 0$.
	If for some $k' < d$, $\mR_{k'} = \bm{0}$, then from Eq.~\eqref{eq:Rk}, we have $\mR_{k} = \bm{0}, \forall k \geq k'$. 
    Thus, Eq.~\eqref{eq:tau-k} trivially holds for $k \geq k'$.
    
    Now we suppose $\forall 0 \leq k < d, \mR_k \neq \bm{0}$. Then by $\mR_k \succeq \bm{0}$ and $\mR_k \neq \bm{0}$, we must have $\bar\vu_k \not\in \mathrm{Ker}(\mR_k)$ in view of Eq.~\eqref{eq:greedy-sr1} for all $0 \leq k < d$.
    Additionally, $\bar\vu_i \in \mathrm{Ker}(\mR_k), \forall 0\leq i < k$ by 1) and 2), so we can see $\bar\vu_i \neq \bar\vu_k, \forall i \neq k$. 
    Thus, at least $k$ of the diagonal elements of $\mR_k$ must be zero, leading to
    \begin{equation}\label{eq:uru}
        \max_{\vu \in \{\ve_1, \dots, \ve_d\}} \vu^\top \mR_k\vu \geq \frac{1}{d-k} \cdot \tr(\mR_k) = \frac{1}{d-k} \cdot \tau_k.
    \end{equation}
    Finally, for all $0 \leq k < d$, since $\bar\vu_k \not \in \mathrm{Ker}(\mR_k)$, then Eq.~\eqref{eq:tau_update} is well-defined. Thus, we get
    \begin{align*}
        &\tau_{k+1} \stackrel{\eqref{eq:tau_update}}{=} \tau_k - \frac{\bar\vu_k^\top \mR_k^2 \bar\vu_k}{\bar\vu_k^\top\mR_k\bar\vu_k} \stackrel{\eqref{eq:cauchy}}{\leq} \tau_k - \frac{\bar\vu_k^\top\mR_k\bar\vu_k}{\bar\vu_k^\top\bar\vu_k}
        \stackrel{\eqref{eq:greedy-sr1}}{=} \tau_k - \max_{\vu \in \{\ve_1, \dots, \ve_d\}} \vu^\top \mR_k\vu \stackrel{\eqref{eq:uru}}{\leq} \left(1-\frac{1}{d-k}\right) \tau_k.
    \end{align*}
	Consequently, we have for all $1 \leq k \leq d$,
	\begin{equation*}
		\tau_k \leq \frac{d-k}{d-k+1} \tau_{k-1} \leq \cdots \leq \left[\prod_{j=1}^{k} \frac{d-j}{d-j+1}\right] \tau_0 = \left(1-\frac{k}{d}\right) \tau_0.
	\end{equation*}
	Then $\tau_d=0$, which leads to $\mG_d=\mA$. Further we obtain $\mG_k=\mA, \forall k \geq d$ following Eq.~\eqref{eq:Rk}. Therefore, $\tau_k=0, \forall k \geq d$. We conclude that for all $k\geq 0$, $\tau_k \leq \left(1-\frac{k}{d}\right)_+ \tau_0$.
	
	\textbf{For the random method}, $\forall k \geq 0, \vu_k$s are independently chosen from an identical spherically symmetric distribution, such as $\mathcal{N}(0, \mI_d), \mathrm{Unif}(\mathcal{S}^{d-1})$.
	
	We first consider $0 \leq k < d$. We have $\lambda_i:= \lambda_i(\mR_k) \geq 0$ since $\mR_k \succeq \bm{0}$.
	Suppose $r_k := \mathrm{rank}(\mR_k) \geq 1$, i.e., $\mR_k \neq \bm{0}$.
	We denote $\mR_k = \mU_k \bm{\Lambda}_k \mU_k^\top$ as the spectral decomposition of $\mR_k$ with an orthogonal matrix $\mU_k$ and a diagonal matrix $\bm{\Lambda}_k = \diag\{\lambda_1,\dots, \lambda_{r_k}, 0, \dots, 0\}$, and $\vv_k = (v_1,\dots, v_d)^\top := \mU_k^\top \vu_k$. 
	Then we can derive that
	\begin{align}
    	&\mathbb{E}_{\vu_k} \frac{\vu_k^\top\mR_k^2\vu_k}{\vu_k^\top\mR_k\vu_k} \mathbbm{1}_{\{\mR_k\vu_k \neq \bm{0}\}} \stackrel{(\romannumeral1)}{=} \mathbb{E}_{\vv_k} \frac{\sum_{i=1}^{r_k}\lambda_i^2 v_i^2}{\sum_{i=1}^{r_k}\lambda_i v_i^2} \mathbbm{1}_{\{\bm{\Lambda}_k \vv_k \neq \bm{0}\}} \stackrel{(\romannumeral2)}{\geq} \mathbb{E}_{\vv_k} \frac{\sum_{i=1}^{r_k}\lambda_i v_i^2}{\sum_{i=1}^{r_k} v_i^2} \mathbbm{1}_{\{\sum_{i=1}^{r_k} v_i^2 \neq 0\}} \nonumber \\
    	= \ & \sum_{i=1}^{r_k}\lambda_i \mathbb{E}_{\vv_k}\frac{v_i^2}{\sum_{j=1}^{r_k}v_j^2} \mathbbm{1}_{\{\sum_{i=1}^{r_k} v_i^2 \neq 0\}} \stackrel{(\romannumeral3)}{=} \frac{1}{r_k}\sum_{i=1}^{r_k}\lambda_i = \frac{\tr(\mR_k)}{r_k} \stackrel{\eqref{eq:lam-sig}}{=} \frac{\tau_k}{r_k}, \label{eq:exp-sr1}
	\end{align}
    where $(\romannumeral1)$ holds due to $\vu_k^\top\mR_k^i\vu_k = \vv_k^\top \bm{\Lambda_k}^i \vv_k, \forall i\geq 1$ and $\mR_k\vu_k \neq \bm{0}$ is equivalent to $\bm{\Lambda}_k \vv_k \neq \bm{0}$; $(\romannumeral2)$ holds due to the Cauchy–Schwarz inequality $\left(\sum_{i=1}^{r_k}\lambda_i^2v_i^2 \right)\left(\sum_{i=1}^{r_k}v_i^2\right)\geq \left(\sum_{i=1}^{r_k}\lambda_i v_i^2\right)^2$ with $\sum_{i=1}^{r_k}\lambda_i v_i^2 = \vu_k^\top\mR_k\vu_k>0$ since $\mR_k\vu_k \neq \bm{0}$ and $\mR_k \succeq \bm{0}$, and the fact that
	\[ \bm{\Lambda}_k \vv_k \neq \bm{0} \Leftrightarrow \exists 1 \leq i \leq r_k, s.t., \lambda_iv_i \neq 0  \Leftrightarrow \exists 1 \leq i \leq r_k, s.t., v_i \neq 0 \Leftrightarrow \sum_{j=1}^{r_k} v_j^2 \neq 0; \]
	$(\romannumeral3)$ uses the fact that $\vv_k$ is still spherically symmetric, thus also permutation invariant:
	\begin{align*}
	    &\mathbb{E}_{\vv_k} \frac{v_1^2}{\sum_{j=1}^{r_k}v_j^2}\mathbbm{1}_{\gV_k} = \dots = \mathbb{E}_{\vv_k} \frac{v_{r_k}^2}{\sum_{j=1}^{r_k}v_j^2}\mathbbm{1}_{\gV_k} = \frac{1}{r_k}\sum_{i=1}^{r_k} \mathbb{E}_{\vv_k} \frac{v_i^2}{\sum_{j=1}^{r_k}v_j^2} \mathbbm{1}_{\gV_k} \\
	    =& \frac{1}{r_k} \mathbb{E}_{\vv_k} \sum_{i=1}^{r_k} \frac{v_i^2}{\sum_{j=1}^{r_k}v_j^2} \mathbbm{1}_{\gV_k} = \frac{1}{r_k} \mathbb{E}_{\vv_k} \mathbbm{1}_{\gV_k} \stackrel{(\romannumeral4)}{=} \frac{1}{r_k}, 
	    \text{ with } \gV_k := \{\vv_k: \sum_{i=1}^{r_k} v_i^2 \neq 0\},
	\end{align*}
    where $(\romannumeral4)$ holds because $\gV_k^c$ (the complementary event of $\gV_k$) has zero Lebesgue measure. Therefore, the random choice of $\vu_k$ leads to
	\begin{align*}
	    &\E_{\vu_k} \tau_{k+1} = \E_{\vu_k} \tau_{k+1} \left[\mathbbm{1}_{\{\mR_k\vu_k\neq\bm{0}\}} + \mathbbm{1}_{\{\mR_k\vu_k = \bm{0}\}}\right] \stackrel{(\romannumeral5)}{=} \E_{\vu_k} \tau_{k+1} \mathbbm{1}_{\{\mR_k\vu_k \neq \bm{0}\}} \\ 
	    \stackrel{\eqref{eq:tau_update}}{=} & \tau_{k} - \mathbb{E}_{\vu_k} \frac{\vu_k^\top\mR_k^2\vu_k}{\vu_k^\top\mR_k\vu_k} \mathbbm{1}_{\{\mR_k\vu_k\neq\bm{0}\}} \stackrel{\eqref{eq:exp-sr1}}{\leq} \left(1-\frac{1}{r_k}\right)\tau_k, \text{ if } \mR_k \neq \bm{0},
	\end{align*}
	where $(\romannumeral5)$ uses the fact that $\mR_k\vu_k = \bm{0}$ is equal to $\sum_{i=1}^{r_k} v_i^2 = 0$, i.e., the event $\gV_k^c$, which has zero measure.

	Furthermore, if $\vu_0, \dots, \vu_{k-1}$ are linearly independent, then by 1) and 2), the dimension of $\mathrm{Ker}(\mR_i), i\leq k$ grows at least by one at every iteration, showing that $\mathrm{Ker}(\mR_k) \geq k$ and $r_k = \mathrm{rank}(\mR_k) = d- \mathrm{Ker}(\mR_k) \leq d-k$. Thus, we establish that
	\begin{equation}\label{eq:exp1}
	    \E [\tau_{k+1} | \gM_k, \mR_k\neq \bm{0} ] \leq \left(1-\frac{1}{d-k}\right) \E [\tau_k | \gM_k, \mR_k\neq \bm{0} ],
	\end{equation}
	where $\gM_k = \{ \vu_0, \dots, \vu_{k-1} \text{ are linear independent}\}, k\geq 1$ and $\gM_0$ is the full set.
	Besides, we note that $\mR_k=\bm{0}$ gives $\mR_{k+1}=\bm{0}$, then
	\begin{equation}\label{eq:exp2}
	    \E [\tau_{k+1} | \gM_k, \mR_k =\bm{0} ] = 0 = \left(1-\frac{1}{d-k}\right) \E [\tau_k | \gM_k, \mR_k = \bm{0}].
	\end{equation}
	Using the law of total expectation conditioning on $\gM_k$\footnote{$\E [\mX|\mA, \gM] \cdot \sP(\mA|\gM)+\E [\mX|\mA^c, \gM] \cdot \sP(\mA^c|\gM)=\E [\mX|\gM]$.}, we obtain 
	\begin{equation}\label{eq:exp3}
	    \E [\tau_{k+1} | \gM_k ] \stackrel{\eqref{eq:exp1}\eqref{eq:exp2}}{\leq} \left(1-\frac{1}{d-k}\right) \E [\tau_k | \gM_k ].
	\end{equation}
	Noting that $\sP(\gM_k^c) = \sP(\exists 0 \leq t \leq k-1, \vu_t \in\text{Span}\{\vu_0, \dots, \vu_{t-1}, \vu_{t+1}, \cdots, \vu_{k-1}\})=0$ since the dimension of $\text{Span}\{\vu_0, \dots, \vu_{t-1}, \vu_{t+1}, \cdots, \vu_{k-1}\}$ is at most $k-1<d$, so by the law of total expectation again\footnote{$\E [\mX| \gM] \cdot \sP(\gM)+\E [\mX|\gM^c] \cdot \sP(\gM^c)=\E \mX$.}, we conclude
	\begin{equation}\label{eq:exp4}
	    \E \tau_{k+1} \stackrel{\eqref{eq:exp3}}{\leq} \left(1-\frac{1}{d-k}\right) \E \tau_k, \forall 0 \leq k < d.
	\end{equation}
	Consequently, we have for all $1 \leq k \leq d$,
	\begin{equation*}
		\E \tau_{k} \stackrel{\eqref{eq:exp4}}{\leq} \frac{d-k}{d-k+1} \cdot \E \tau_{k-1} \leq \cdots \leq \left[\prod_{j=1}^{k}\frac{d-j}{d-j+1}\right] \E \tau_{0} = \left(1-\frac{k}{d}\right)\E \tau_0.
	\end{equation*}
	That is, we obtain $\E \tau_{d}=0$, showing that $\tau_d=0$ a.s. and $\mG_d - \mA = \bm{0}$ a.s., since $\mG_d-\mA \succeq \bm{0}$. Furthermore, following Eq.~\eqref{eq:Rk}, we obtain $\forall k\geq d, \mG_k=\mA$ a.s. Therefore, we derive that $\forall k\geq d, \tau_k=0$ a.s. Then we conclude that $\E \tau_{k} \leq \left(1-\frac{k}{d}\right)_+\E \tau_0$.
\end{proof}

\subsection{Proof of Theorem \ref{thm:bfgs-update}}\label{app:miss-prove2}
\begin{proof}
	\textbf{For the greedy method}, at step $k \geq 0$, since $\mG_k^{-1}=\mL_k^\top\mL_k$, we obtain
	\begin{equation}\label{eq:greedybfgs}
	\max_{\tilde\vu\in\{\bm{e}_{1}, \dots, \bm{e}_{d}\}} \tilde{\vu}^\top \mL_k^{-\top}\mA^{-1}\mL_k^{-1}\tilde{\vu} \geq \frac{1}{d}\tr\left(\mL_k^{-\top}\mA^{-1}\mL_k^{-1}\right) = \frac{1}{d}\tr\left(\mL_k^{-1}\mL_k^{-\top}\mA^{-1}\right) = \frac{1}{d} \tr(\mG_k\mA^{-1}).
	\end{equation}
	Therefore, the greedy choice of $\vu_k = \mL_k^\top \tilde\vu_k$ with $\tilde\vu_k$ following Eq.~\eqref{eq:greedy-bfgs} leads to
	\begin{equation*}
	    \begin{aligned}
	        \sigma_{k+1} &\stackrel{\eqref{eq:sigma-up2}}{=} \sigma_k - \frac{\tilde\vu_k^\top \mL_k^{-\top} \mA^{-1}\mL_k^{-1}\tilde\vu_k}{\tilde\vu_k^\top \tilde\vu_k} + 1 \stackrel{\eqref{eq:greedy-bfgs}\eqref{eq:greedybfgs}}{\leq} 
	        \sigma_{k} - \frac{1}{d} \tr(\mG_k\mA^{-1}) + 1 \\
	        &\ = \left(1-\frac{1}{d}\right)\sigma_k \leq \cdots \leq  \left(1-\frac{1}{d}\right)^{k+1} \sigma_0, \forall k \geq 0.
	\end{aligned}
	\end{equation*}

	\textbf{For the random method}, we have that $\mathbb{E}_{\tilde\vu} \frac{\tilde\vu\tilde\vu^\top}{\tilde\vu^\top\tilde\vu} = \frac{1}{d}\mI_d$. Hence, we obtain
	\begin{equation}\label{eq:random-bfgs-cal}
	\begin{aligned}
	\mathbb{E}_{\tilde\vu} \frac{\tilde\vu^\top \mL_k^{-\top}\mA^{-1}\mL_k^{-1} \tilde\vu}{\tilde\vu^\top\tilde\vu} &= \tr\left[\mL_k^{-\top}\mA^{-1}\mL_k^{-1} \cdot 
	\mathbb{E}_{\tilde\vu} \frac{\tilde\vu\tilde\vu^\top}{\tilde\vu^\top\tilde\vu}
	\right] {=} \frac{1}{d}\tr(\mL_k^{-\top}\mA^{-1}\mL_k^{-1}) {=} \frac{1}{d}\tr(\mG_k\mA^{-1}),
	\end{aligned}
	\end{equation}
	Therefore, the random choice of $\vu_k = \mL_k^\top \tilde\vu_k$ leads to
	\begin{align*}
	\E_{\vu_k} \sigma_{k+1} &\stackrel{\eqref{eq:sigma-up2}}{=} \sigma_k - \E_{\tilde\vu_k}\frac{\tilde\vu_k^\top \mL_k^{-\top} \mA^{-1}\mL_k^{-1}\tilde\vu_k}{\tilde\vu_k^\top \tilde\vu_k} + 1 \stackrel{\eqref{eq:random-bfgs-cal}}{=} \sigma_k-\frac{1}{d}\tr(\mG_k\mA^{-1}) + 1 = \left(1-\frac{1}{d}\right) \sigma_{k},
	\end{align*}
	Now taking expectation for all $\{\vu_k\}$, we get 
	\begin{equation*}
	\E \sigma_{k+1} = \left(1-\frac{1}{d}\right)\E \sigma_{k} = \dots =   \left(1-\frac{1}{d}\right)^{k+1}\E\sigma_0, \forall k \geq 0.
	\end{equation*}
\end{proof}

\section{Missing Proofs of Quadratic Objective}

\subsection{Proofs of Theorem \ref{thm:random-quad} and Theorem \ref{thm:quad}}\label{app:miss-prove-random2}
\begin{proof}
    From Lemma \ref{lemma:monotonic}, we have $\forall k\geq 0, \mG_k\succeq\mA$. 
    Now for all $k \geq 0$, we denote $\eta_k := \norm{\mA^{-1/2}\mG_k\mA^{-1/2}}$, then $\eta_k\geq 1$ and $\mG_k \preceq \eta_k \mA$. Hence,
	\begin{eqnarray}
	\eta_k-1 = \norm{\mA^{-1/2}(\mG_k-\mA)\mA^{-1/2}} &\leq&  \tr\left(\mA^{-1/2}\left(\mG_k-\mA\right)\mA^{-1/2} \right) \stackrel{\eqref{eq:sigmaA}}{=} \sigma_{k} \label{eq:eta-1} \\ 
	&\stackrel{\eqref{eq:object-f}}{\leq}& \frac{\tr\left(\mG_k-\mA\right)}{\mu} \stackrel{\eqref{eq:tauA}}{=} \frac{\tau_{k}}{\mu}. \label{eq:eta-2} 
	\end{eqnarray}
	By Lemma \ref{lemma:lambda}, we know that for all $k \geq 0$, $\lambda_{k+1} \leq (\eta_k-1)\lambda_k$, then we can take $\rho_k = \eta_k-1$.
	
    \begin{enumerate}
        \item For Broyden method, using Theorem \ref{thm:rand-update}, we get
        \[ \E \rho_k = \E (\eta_k-1) \stackrel{\eqref{eq:eta-1}}{\leq} \E \sigma_{k} \stackrel{\eqref{eq:random-k}}{\leq} \left(1-\frac{1}{d\varkappa}\right)^k \E\sigma_0, \forall k \geq 0. \]
        \item For SR1 method, using Theorem \ref{thm:sr1-update}, we obtain
    	\begin{equation*}
    	\E \rho_k = \E (\eta_k-1) \stackrel{\eqref{eq:eta-2}}{\leq} \frac{\E \tau_{k}}{\mu} \stackrel{\eqref{eq:tau-k}}{\leq} \left(1-\frac{k}{d}\right)_+ \frac{\E\tau_0}{\mu}, \forall k \geq 0.
    	\end{equation*}
    	\item For BFGS method, using Theorem \ref{thm:bfgs-update}, we get
    	\[ \E \rho_k = \E(\eta_k-1) \stackrel{\eqref{eq:eta-1}}{\leq} \E \sigma_{k} \stackrel{\eqref{eq:sigma-k}}{\leq} \left(1-\frac{1}{d}\right)^k \E\sigma_{0}, \forall k \geq 0. \]
    \end{enumerate}
\end{proof}

\subsection{Proof of Corollary \ref{cor:quad-ranom-p}}\label{app:miss-prove-random2.5}
\begin{proof}
    From Theorems \ref{thm:rand-update} and \ref{thm:random-quad}, we can apply Lemma \ref{lemma:p} with $\mX_k=\sigma_k$ or $\rho_k, \forall k \geq 0$ and $a = \E\sigma_0, t = d\varkappa$, i.e., with probability at least $1-\delta/2$, we get
	\begin{equation}\label{eq:sigma-again}
	\sigma_k \leq \frac{2d^2\varkappa^2\E\sigma_0}{\delta} \left(1-\frac{1}{d\varkappa+1}\right)^{k}, \forall k \geq 0,
	\end{equation}
	and with probability at least $1-\delta/2$, we have
	\begin{equation}\label{eq:rho_tel}
	    \rho_k \leq \frac{2d^2\varkappa^2\E\sigma_0}{\delta} \left(1-\frac{1}{d\varkappa+1}\right)^{k}, \forall k \geq 0. 
	\end{equation}
	Noting that $\lambda_{k+1} \leq \rho_k \lambda_k$ by definition, we furtehr obtain with probability at least $1-\delta/2$, 
	\begin{equation}\label{eq:lam_tel}
	    \lambda_{k+1} \leq \frac{2d^2\varkappa^2\E\sigma_0}{\delta} \left(1-\frac{1}{d\varkappa+1}\right)^{k} \lambda_k, \forall k \geq 0, 
	\end{equation}   
	because the latter event (as a set) in Eq.~\eqref{eq:lam_tel} is contained in the former in Eq.~\eqref{eq:rho_tel}.
	
  	Using the fact that for any sequences $\{\bm{\xi}_k\}$ and $\{a_k\}$ of nonnegative random variables and nonnegative reals, respectively, it holds that 
  	\begin{equation}\label{eq:fact}
        \sP\left(\bm{\xi}_{k+1} \leq \left(\prod_{i=0}^k a_i\right) \bm{\xi}_0, \forall k \geq 0 \right) \geq \sP\left(\bm{\xi}_{k+1} \leq a_k \bm{\xi}_k, \forall k \geq 0 \right),
  	\end{equation}
  	because the latter event (as a set) is contained in the former.
  	Thus we can telescope from $k-1$ to $0$ in Eq.~\eqref{eq:lam_tel} for all $k \geq 1$.
  	Then we get with probability at least $1-\delta/2$, 
	\begin{equation}\label{eq:lam-again}
	    \lambda_{k} \stackrel{\eqref{eq:lam_tel}\eqref{eq:fact}}{\leq} \left(\frac{2d^2\varkappa^2\E\sigma_0}{\delta}\right)^{k}\left(1-\frac{1}{d\varkappa+1}\right)^{k(k-1)/2} \lambda_0, \forall k \geq 1,
	\end{equation}
	and the above inequality trivially holds for $k=0$.
	
	Finally, applying the union bound again makes both Eqs.~\eqref{eq:sigma-again} and \eqref{eq:lam-again} hold with probability at least $1-\delta$.
\end{proof}

\section{Missing Proofs of General Objective}

\subsection{Proofs of Lemma \ref{lemma:random-gel} and Lemma \ref{lemma:gel-bfgs-sr1}}\label{app:miss-prove-random3}
Since the proofs of Lemma \ref{lemma:random-gel} and Lemma \ref{lemma:gel-bfgs-sr1} have many overlapping parts, we recombine them into a lemma below.
\begin{lemma}[Restating]\label{lemma:com-exp}
	Suppose in Algorithm \ref{algo:general-random} or Algorithm \ref{algo:general-quasi}, a random initialization $\mG_0$ always satisfies $\nabla^2 f(\vx_0) \preceq \mG_0 \preceq \eta \nabla^2 f(\vx_0)$ for some $\eta \geq 1$, and the (random) initial point $\vx_0$ is sufficiently close to the solution:
	\begin{equation}\label{eq:init}
	    M\lambda_0 \leq \frac{\ln2}{4\eta(2 c d+1)}.
	\end{equation}
	Then for all $k \geq 0$, we have $\nabla^2 f(\vx_k) \preceq \mG_k \preceq (1+\delta_k)\nabla^2 f(\vx_k)$, where $\delta_k$ is a certain random variable such that
	\begin{equation}\label{eq:delta}
	    \E \delta_k \leq 2c d \eta \left(1-\frac{1}{t}\right)^k,
	\end{equation}
	and $\lambda_{k+1} \leq \rho_k\lambda_k$, where $\rho_k$ is a certain random variable such that
	\begin{equation}\label{eq:lam}
	    \E \rho_k \leq 2c d\eta \left(1-\frac{1}{t}\right)^k.
	\end{equation}
	Here, the choices of $\delta_k,t,c$ are inherited from Lemma \ref{lemma:general-update}.
\end{lemma}

\begin{proof}
	The derivation is the same as Theorem 4.9 in the work of  \citet{rodomanov2021greedy} by using Lemma \ref{lemma:general-update}. In order to providing better paper readability, we also show the detail below. 
	
	In view of Theorem \ref{app:aux-thm}, because the initial condition $\frac{\ln2}{4\eta(2c d + 1)}\leq \frac{\ln\frac{3}{2}}{4\eta}$,  we get $\nabla^2 f(\vx_k) \preceq \mG_k, \forall k \geq 0$, and also 
	\begin{equation}\label{eq:m-lam-ran}
	M\sum_{i=0}^k \lambda_i \stackrel{\eqref{eq:lambda-a}}{\leq} M\lambda_0\sum_{i=0}^{k}\left(1-\frac{1}{2\eta}\right)^i \leq 2\eta M \lambda_0 \stackrel{\eqref{eq:init}}{\leq} \frac{\ln2}{2(2 c d+1)}.
	\end{equation}
	Moreover, we need to underline that Eq.~\eqref{eq:m-lam-ran} holds with no expectation, which is crucial in the following derivation.
	Next, let us show that $\forall k\geq 0$,
	\begin{equation}\label{eq:sig-lam-ran}
		\E \theta_k \leq 2c d\eta\left(1-\frac{1}{t}\right)^k, \ \mathrm{ where } \ \theta_k := \delta_k+2cd M\lambda_k.
	\end{equation}
	Indeed, because $\nabla^2 f(\vx_0) \preceq \mG_0 \preceq \eta \nabla^2 f(\vx_0)$, we have
	\[ \E \sigma_0 \stackrel{\eqref{eq:lam-sig}}{=} \E \tr\left(\nabla^2 f(\vx_0)^{-1}\mG_0\right)-d \leq d\left(\eta-1\right), \]
	and
	\[ \E \; \frac{d\varkappa \tau_0}{\tr[\nabla^2 f(\vx_0)]} \stackrel{\eqref{eq:lam-sig}}{=}  d\varkappa \cdot \frac{\E \tr[\mG_0-\nabla^2 f(\vx_0)]}{\tr[\nabla^2 f(\vx_0)]} \leq \varkappa d (\eta-1). \]
	Hence, following the choice of $\delta_0$ from Lemma \ref{lemma:general-update}, we derive that
	\begin{equation}\label{eq:init-theta-ran}
    	\E \theta_0 = \E \delta_0+2c d M\lambda_0
    	\stackrel{\eqref{eq:init}}{\leq} c d \left(\eta-1\right) + \frac{2c d }{2c d + 1} \cdot \frac{\ln2}{4\eta} < c d \left(\eta-1\right) + 1 \leq c d \eta.
	\end{equation} 
	Therefore, for $k=0$, Eq.~\eqref{eq:sig-lam-ran} is satisfied. 
	
	Now we consider the index $k+1 \geq 1$.
	Because Eq.~\eqref{eq:m-lam-ran} shows that $M\lambda_k\leq 2$, we can employ Lemma \ref{lemma:op-ud}, which leads to
	\begin{equation}\label{eq:r-lam-ran}
		r_k:=\|\vx_{k+1}-\vx_k\|_{\vx_k} \stackrel{\eqref{eq:lam-eta}}{\leq} \lambda_k.
	\end{equation}
	Then by Lemma \ref{lemma:general-update}, we have
	\begin{equation}\label{eq:delta-k-again}
	    \forall k\geq 0, \nabla^2 f(\vx_k) \preceq \mG_k \preceq (1+\delta_k)\nabla^2 f(\vx_k)
	\end{equation}
	with 
	\begin{equation}\label{eq:sigma-k-up-all}
    	\E_{\vu_k} \delta_{k+1} \stackrel{\eqref{eq:delta-up}}{\leq} \left(1-\frac{1}{t}\right) \left(1+Mr_k\right)^2
    	\left(\delta_k+2c d M r_k\right)\stackrel{\eqref{eq:r-lam-ran}}{\leq} \left(1-\frac{1}{t}\right)e^{2M\lambda_k}\theta_k.
	\end{equation}
	Moreover, using Lemma \ref{lemma:op-ud} again, we obtain
	\begin{equation}\label{eq:lam-re-ran}
    	\lambda_{k+1} \stackrel{\eqref{eq:lam-eta}\eqref{eq:delta-k-again}}{\leq} \left(1+\frac{M\lambda_k}{2}\right) \frac{\delta_k+\frac{M\lambda_k}{2}}{1+\delta_k}\lambda_k \leq \left(1+\frac{M\lambda_k}{2}\right) \theta_k \lambda_k \leq e^{2M\lambda_k}\theta_k\lambda_k. 
	\end{equation}
	Note that $\frac{1}{2} \leq 1-\frac{1}{t}$ because $t \geq d \geq 2$. 
	Combining Eqs.~\eqref{eq:sigma-k-up-all} and \eqref{eq:lam-re-ran}, we obtain
	\begin{align*}
    	\E_{\vu_k} \theta_{k+1} &\leq \left(1-\frac{1}{t}\right)e^{2M\lambda_k}\theta_k+ 2c d M e^{2M\lambda_k} \theta_k \lambda_k \\
    	&\leq \left[\left(1-\frac{1}{t}\right) + \left(1-\frac{1}{t}\right) 4c d M \lambda_k\right]e^{2M\lambda_k}\theta_k = \left(1-\frac{1}{t}\right)\left(1+4c d M\lambda_k\right)e^{2M\lambda_k}\theta_k \\
        &\leq \left(1-\frac{1}{t}\right)e^{2(2c d+1)M\lambda_k}\theta_k \stackrel{\eqref{eq:lambda-a}}{\leq} \left(1-\frac{1}{t}\right)e^{2(2cd+1) M\lambda_0\left(1-\frac{1}{2\eta}\right)^k}\theta_k.
    \end{align*}
	Therefore, by taking expectation for all randomness, we obtain
	\begin{align}
		\E \theta_{k+1} & \leq \left(1-\frac{1}{t}\right) e^{2(2c d+1) M\lambda_0\left(1-\frac{1}{2\eta}\right)^k} \E \theta_{k} \leq \cdots \stackrel{(*)}{\leq} \left(1-\frac{1}{t}\right)^{k+1} e^{2(2c d+1) M\lambda_0\sum_{i=0}^k\left(1-\frac{1}{2\eta}\right)^i} \E \theta_{0} \nonumber \\ 
		& \leq \left(1-\frac{1}{t}\right)^{k+1} e^{4\eta(2c d+1) M\lambda_0} \E \theta_{0} \stackrel{\eqref{eq:m-lam-ran}\eqref{eq:init-theta-ran}}{\leq} 2c d\eta\left(1-\frac{1}{t}\right)^{k+1}. \label{eq:theta-2-ran}
	\end{align}
	Thus, Eq.~\eqref{eq:sig-lam-ran} is proved for the index $k+1$. 
	Therefore, Eq.~\eqref{eq:sig-lam-ran} holds for all $k \geq 0$.
	
	Now we prove the bound of $\delta_k$ and $\rho_k$ based on Eq.~\eqref{eq:sig-lam-ran}.
	Since $\lambda_k\geq 0$, we have
	\[ \E \delta_k \leq \E \delta_k+2c d M\lambda_k \stackrel{\eqref{eq:sig-lam-ran}}{=} \E \theta_k \stackrel{\eqref{eq:sig-lam-ran}}{\leq} 2c d\eta\left(1-\frac{1}{t}\right)^k, \forall k \geq 0. \]
	Finally, we adopt $\rho_k=e^{2M\lambda_k}\theta_k$ based on Eq.~\eqref{eq:lam-re-ran} for all $k \geq 0$. Then we get
	\begin{eqnarray*}
		\E \rho_k &=& \E e^{2M \lambda_k}\theta_k \leq \E e^{2(2c d+1)M\lambda_k}\theta_k \stackrel{\eqref{eq:lambda-a}}{\leq} e^{2(2c d+1)M\lambda_0 \left(1-\frac{1}{2\eta}\right)^k} \E \theta_k \\ 
		&\stackrel{\eqref{eq:theta-2-ran}}{\leq}& \left(1-\frac{1}{t}\right)^k e^{2(2c d+1) M\lambda_0 \sum_{i=0}^{k}\left(1-\frac{1}{2\eta}\right)^i} \theta_{0} \stackrel{\eqref{eq:m-lam-ran}\eqref{eq:init-theta-ran}}{\leq} 2c d\eta\left(1-\frac{1}{t}\right)^k,
	\end{eqnarray*}
	where we use the inequality $(*)$ in Eq.~\eqref{eq:theta-2-ran} with subscript $k-1$.
	Thus, Eqs.~\eqref{eq:delta} and \eqref{eq:lam} are proved.
\end{proof}

\subsection{Proofs of Theorem \ref{thm:random-p} and Theorem \ref{thm:gre-ran-p}}\label{app:miss-prove-random4}
\begin{proof}
    The results of greedy methods directly follow Lemma \ref{lemma:gel-bfgs-sr1},
	and the proof of randomized methods is the same as the proof of Corollary \ref{cor:quad-ranom-p} by applying Lemma \ref{lemma:p} with $\mX_k=\delta_k$ or $\rho_k, \forall k \geq 0$, so we omit it.
\end{proof}

\subsection{Proofs of Corollary \ref{cor:random-gel} and Corollary \ref{cor:gel-sr1-bfgs} }\label{app:miss-prove-random5}
Since the proofs of Corollary \ref{cor:random-gel} and Corollary \ref{cor:gel-sr1-bfgs} also have many overlapping parts, we also recombine them into a corollary below.

\begin{corollary}[Restating]
	Suppose in Algorithm \ref{algo:general-random} or Algorithm \ref{algo:general-quasi}, $\mG_0=L \mI_d$ and $\vx_0$ satisfies $M\lambda_0 \leq (\ln\frac{3}{2})/(4\varkappa)$.
	Then we could obtain for randomized methods, with probability at least $1-\delta$ over the random directions $\{\vu_k\}$,
	\[ \lambda_{k_0+k} \leq \left(1-\frac{1}{t+1}\right)^{k(k-1)/2} \cdot \left(\frac{1}{2}\right)^k \cdot \left(1-\frac{1}{2\varkappa}\right)^{k_0} \cdot \lambda_0, \; \forall k \geq 0, \]
	where $k_0 = O\left((t+\varkappa)\ln(c d\varkappa t/\delta)\right)$,
	and for greedy methods,
	\[ \lambda_{k_0+k} \leq \left(1-\frac{1}{t}\right)^{k(k-1)/2} \cdot \left(\frac{1}{2}\right)^k \cdot \left(1-\frac{1}{2\varkappa}\right)^{k_0} \cdot \lambda_0, \text{ for all } k \geq 0, \]
	where $k_0=O\left((t+\varkappa)\ln(c d\varkappa)\right)$.
	Here, the choices of $t, c$ are inherited from Lemma \ref{lemma:general-update}.
\end{corollary}

\begin{proof}
	We note that $\mG_0 =L\mI_d$ gives $\nabla^2 f(\vx_0) \preceq \mG_0 \preceq \varkappa \nabla^2 f(\vx_0)$.
	Since the initial point $\vx_0$ is sufficiently close to the solution: $M\lambda_0 \leq \frac{\ln\frac{3}{2}}{4\varkappa}$, Theorem \ref{app:aux-thm} holds for $\eta=\varkappa$.
	
	Denote by $k_1\geq 0$ the number of the first iteration, for which
	\[ \left(1-\frac{1}{2\varkappa}\right)^{k_1} \leq \frac{2}{3} \cdot \frac{1}{2c d+1}. \]
	Clearly, $k_1 \leq 2 \varkappa \ln(3c d+2) + 1$. 
	Then from Theorem \ref{app:aux-thm}, we obtain
	\begin{equation*}
	    \nabla^2 f(\vx_{k_1}) \preceq \mG_{k_1} \preceq \frac{3\varkappa}{2} \nabla^2 f(\vx_{k_1}) \text{ and } M\lambda_{k_1} \stackrel{\eqref{eq:lambda-a}}{\leq} M\left(1-\frac{1}{2\varkappa}\right)^{k_1}\lambda_0 \leq \frac{\ln2}{6\varkappa\left(2c d+1\right)},
	\end{equation*}
	which satisfies the initial condition in Lemma \ref{lemma:com-exp} with $\eta=\frac{3}{2} \varkappa$.
    That is, $k_1$ is the number of iterations for entering the region of superlinear convergence.
	Hence, from the (random) initialization $\mG_{k_1}$ and $\vx_{k_1}$, by Lemma \ref{lemma:com-exp}, we have $\forall k \geq k_1, \lambda_{k+1} \leq \rho_k \lambda_k$ with 
	\[ \E \rho_k \leq 3c d\varkappa \left(1-\frac{1}{t}\right)^{k-k_1}. \]
	For randomized method, applying Lemma \ref{lemma:p} with $\mX_k=\rho_k, \forall k \geq 0$ and $a = 3c d\varkappa $, we can obtain with probability at least $1-\delta$, 
	\begin{equation*}
	    \rho_k \leq \frac{3c d \varkappa t^2}{\delta} \left(1-\frac{1}{t+1}\right)^{k-k_1}, \forall k \geq k_1, 
	\end{equation*}
	which leads to with probability at least $1-\delta$, 
    \begin{equation}\label{eq:lam-ran-again}
	    \lambda_{k+1} \leq \rho_k\lambda_k \leq \frac{3c d \varkappa t^2}{\delta} \left(1-\frac{1}{t+1}\right)^{k-k_1}\lambda_k, \forall k \geq k_1. 
	\end{equation}
	Denote by $k_2\geq 0$ the number of the first iteration, for which
	\[ \frac{3c d \varkappa t^2}{\delta}\left(1-\frac{1}{t+1}\right)^{k_2} \leq \frac{1}{2}. \]
	Clearly, $k_2 \leq (t+1) \ln(6c d \varkappa t^2/\delta)+1$, which is the number of iterations to make the superlinear rate `valid'.
	Applying Eq.~\eqref{eq:lam-ran-again} only to all $k \geq k_1+k_2$ (which includes the event in Eq.~\eqref{eq:lam-ran-again}), we still get with probability at least $1-\delta$,
	\[ \lambda_{k_1+k_2+k+1} \leq \frac{3c d \varkappa t^2}{\delta} \left(1 {-} \frac{1}{t {+} 1}\right)^{k_2+k} \lambda_{k_1+k_2+k} \leq \left(\frac{1}{2}\right) \cdot \left(1 {-} \frac{1}{t {+} 1}\right)^{k} \lambda_{k_1+k_2+k}, \forall k \geq 0. \]
	Therefore, by the fact of Eq.~\eqref{eq:fact}, we also have with probability at least $1-\delta$,
	\begin{equation}\label{eq:p-lam}
	    \lambda_{k_1+k_2+k} \leq \left(1-\frac{1}{t+1}\right)^{k(k-1)/2} \left(\frac{1}{2}\right)^k \, \lambda_{k_1+k_2},  \forall k \geq 0.
	\end{equation}
	Moreover, using Theorem \ref{app:aux-thm} again, we have the deterministic result that
	\begin{equation}\label{eq:k1k2}
	    \lambda_{k_1+k_2} \stackrel{\eqref{eq:lambda-a}}{\leq} \left(1-\frac{1}{2\varkappa}\right)^{k_1+k_2}\lambda_0.
	\end{equation}
	Noting that the event of Eq.~\eqref{eq:p-lam} is contained in the following event based on Eq.~\eqref{eq:k1k2}, we finally obtain with probability at least $1-\delta$, 
	\[ \lambda_{k_0+k} \leq  \left(1-\frac{1}{t+1}\right)^{k(k-1)/2} \cdot \left(\frac{1}{2}\right)^k  \cdot \left(1-\frac{1}{2\varkappa}\right)^{k_0}\lambda_0, \forall k\geq 0, \]
	where $k_0 = k_1 + k_2 \leq 2 \varkappa \ln(3c d+2) +(t+1) \ln(6c d \varkappa t^2/\delta)+2 =  O\left((t+\varkappa)\ln(c d\varkappa t/\delta)\right)$.
	
	Similarly, we can get the results for greedy methods with $k_2 \leq t \ln(6c d \varkappa)+1$, leading to $k_0=k_1+k_2 \leq 2 \varkappa \ln(3c d+2) + t \ln(6c d \varkappa)+2=O\left((t+\varkappa)\ln(c d\varkappa)\right)$.
\end{proof}

Particularly, following the choices of $t,c$ in Lemma \ref{lemma:general-update}, for the random Broyden method in Algorithm \ref{algo:general-random}, one has $t=d\varkappa, c=1$, which leads to $k_0=O\left(d\varkappa\ln(d\varkappa/\delta)\right)$;
for the greedy/random BFGS method in Algorithm \ref{algo:general-quasi}, one has $t=d, c=1$, which leads to $k_0=O\left((d+\varkappa)\ln(d\varkappa)\right)$ and $O\left((d+\varkappa)\ln(d\varkappa/\delta)\right)$; 
for the greedy/random SR1 method in Algorithm \ref{algo:general-quasi}, one has $t=d, c=\varkappa$, which leads to $k_0=O\left((d+\varkappa)\ln(d\varkappa)\right)$ and $O\left((d+\varkappa)\ln(d\varkappa/\delta)\right)$.
	
\section{Proof of Proposition \ref{prop:l}}\label{app:propl}
\begin{proof}
    Denote $\mH_k:=\mG_k^{-1}, \forall k\geq 0$.
    From the inverse update rule in Eq.~\eqref{eq:invbfgs}, we  obtain
    \[ \mH_{k+1} = \mQ_k^\top \mH_k\mQ_k = \left(\mL_k\mQ_k\right)^\top \left(\mL_k\mQ_k\right), \; \mQ_k := \mI_d-\dfrac{\mA \vu_k\vu_k^\top}{\vu_k^\top \mA\vu_k}+\frac{\mG_k\vu_k\vu_k^\top}{\sqrt{\vu_k^\top\mA \vu_k \cdot \vu_k^\top\mG_k\vu_k}}. \]
    Indeed, we have 
    \begin{eqnarray*}
    \lefteqn{\mQ_k^\top \mH_k\mQ_k - \left(\mI_d-\frac{\vu_k\vu_k^\top\mA}{\vu_k^\top\mA\vu_k}\right) \mH_k \left(\mI_d-\frac{\mA \vu_k\vu_k^\top}{\vu_k^\top \mA\vu_k}\right)} \\
    &= & 2\left(\mI_d-\frac{\vu_k\vu_k^\top\mA}{\vu_k^\top\mA\vu_k}\right) \cdot \frac{\mH_k \mG_k\vu_k\vu_k^\top}{\sqrt{\vu_k^\top\mA \vu_k \cdot \vu_k^\top\mG_k\vu_k}} + \frac{\vu_k\vu_k^\top\mG_k\mH_k\mG_k\vu_k\vu_k^\top}{\vu_k^\top\mA\vu_k\cdot \vu_k^\top\mG_k\vu_k} \\
    &=  & 2\left(\mI_d-\frac{\vu_k\vu_k^\top\mA}{\vu_k^\top \mA\vu_k}\right) \vu_k \cdot \frac{\vu_k^\top}{\sqrt{\vu_k^\top\mA \vu_k \cdot \vu_k^\top\mG_k\vu_k}} + \frac{\vu_k\vu_k^\top}{\vu_k^\top\mA\vu_k} =\frac{\vu_k\vu_k^\top}{\vu_k^\top\mA\vu_k},
    \end{eqnarray*}
    which is identical to Eq.~\eqref{eq:invbfgs}. 
    
    Next, note that $\mH_k \succ 0$. Thus the square matrix $\mL_k$ is also nonsingular, leading to $\mL_k\mG_k \mL_k^\top =\mL_k\mL_k^{-1}\mL_k^{-\top} \mL_k^\top=\mI_d$.
    Hence, from $\vu_k=\mL_k^\top \tilde\vu_k$, we get $\mL_k\mG_k \vu_k = \mL_k\mG_k \mL_k^\top\tilde\vu_k=\tilde\vu_k$, and 
    $\vu_k^\top\mG_k \vu_k = \tilde\vu_k^\top \mL_k\mG_k \mL_k^\top\tilde\vu_k=\norm{\tilde\vu_k}^2$.
    Therefore, by the expression of $\vv_k$ in Proposition \ref{prop:l}, we obtain
    \begin{align*}
        \mL_{k+1} &= \mL_k\mQ_k = \mL_k - \frac{\mL_k\mA\vu_k}{\vu_k^\top\mA \vu_k} \cdot \vu_k^\top + 
        \frac{\mL_k\mG_k\vu_k\vu_k^\top}{\sqrt{\vu_k^\top\mA \vu_k \cdot \vu_k^\top \mG_k\vu_k}} \\
        &= \mL_k - \mL_k\mA\vu_k \cdot \frac{\vu_k^\top}{\vu_k^\top\mA \vu_k} + 
        \frac{\tilde\vu_k}{\norm{\tilde\vu_k}} \cdot \sqrt{\vu_k^\top \mA\vu_k} \cdot \frac{\vu_k^\top}{\vu_k^\top\mA \vu_k} = \mL_k - \frac{\left(\mL_k\mA\vu_k-\vv_k\right)\vu_k^\top}{\vu_k^\top\mA \vu_k}.
    \end{align*}
\end{proof}

\vskip 0.2in
\bibliography{reference}

\end{document}